\documentclass[a4paper,oneside, 11pt]{amsart}

\usepackage{amssymb}
\usepackage{amsmath}
\usepackage{subfigure}
\usepackage{epsfig}
\usepackage[all]{xy}
\usepackage[utf8]{inputenc}
\usepackage{mathabx}
\usepackage{graphicx}
\usepackage{xcolor}
\usepackage{ulem}
\usepackage{comment}
\usepackage{stmaryrd}
\usepackage{enumitem}
\usepackage{dynkin-diagrams}
\usepackage{hyperref} 
\hypersetup{colorlinks=true,linkcolor=blue,citecolor=red,urlcolor=blue}

\newtheorem{theorem}{Theorem}[section] 
\newtheorem{lemma}[theorem]{Lemma}     
\newtheorem{corollary}[theorem]{Corollary}
\newtheorem{proposition}[theorem]{Proposition}

\theoremstyle{definition}
\newtheorem{example}[theorem]{Example}
\newtheorem{remark}[theorem]{Remark}

\newtheorem{definition}[theorem]{Definition}

\title[Relative homology of arithmetic subgroups of $\mathrm{SU}(3)$]{Relative homology of arithmetic subgroups of $\mathrm{SU}(3)$}

\begin{document}
\maketitle

\begin{center}
\author{\textbf{Claudio Bravo}
\footnote{
Centre de Mathématiques Laurent Schwartz, École Polytechnique, 
Institut Polytechnique de Paris, 91128 Palaiseau Cedex, France. 
Email: \email{claudio.bravo-castillo@polytechnique.edu}}}
\end{center}

\begin{abstract}
Let $\mathcal{C}$ be a smooth, projective and geometrically integral curve defined over a finite field $\mathbb{F}$.
Let $A$ be the ring of function of $\mathcal{C}$ that are regular outside a closed point $P$ and let $k=\mathrm{Quot}(A)$.
Let $\mathcal{G}=\mathrm{SU}(3)$ be the non-split group-scheme defined from an (isotropic) hermitian form in three variables.
In this work, we describe, in terms of the Euler-Poincar\'e characteristic, the relative homology groups of certain arithmetic subgroups $G$ of $\mathcal{G}(A)$ modulo a representative system $\mathfrak{U}$ of the conjugacy classes of their maximal unipotent subgroups.
In other words, we measure how far are the homology groups of $G$ from being the coproducts of the corresponding homology groups of the subgroups $U \in \mathfrak{U}$.\\

\textbf{MSC Codes:} 55N10, 20G30, 11R58 (primary), 20E08, 20E08, 20E45 (secondary) \\
\textbf{Keywords:} Relative homology, algebraic function fields, special unitary groups, arithmetic subgroups, maximal unipotent subgroups, Bruhat-Tits trees, quotient graphs.
\end{abstract}

\section{Introduction}\label{section introduction}

Let $\mathcal{C}$ be a smooth, projective and geometrically integral curve defined over a field $\mathbb{F}$. Let $P$ be a closed point in $\mathcal{C}$, and let us denote by $A$ the ring of functions that are regular on $\mathcal{C}$ outside $P$. Let us write $k=\mathrm{Quot}(A)$.

Discrete subgroups of a real Lie group $G$ have been studied by their action on symmetric spaces, that is, on homogeneous spaces of the form $\mathcal{X}= K \backslash G$, where $K$ is a maximal compact subgroup of $G$. 
In analogy with the preceding approach, Bruhat and Tits in \cite{BruhatTits1} and \cite{BruhatTits2} establish a program for study groups of $k$-points of reductive quasi-split $k$-group schemes through their action on simplicial complexes called buildings. 
This point of view has been fruitful in the past to study groups of arithmetical interest. For instance, \cite[Theorem 3]{So} presents the group of $\mathbb{F}[t]$-points of a split group-scheme as an amalgamated free product of simpler subgroups. Then \cite[Theorem 5]{So} describes the homology of such groups. The preceding presentations have been extended to the context of isotrivial $k$-group schemes in \cite[Theorem 3.9]{Margaux}.
One of the strongest results about the actions of arithmetical groups on buildings can be found in \cite[Prop. 13.6]{Bux-G-Witzel}.

In the context of group-schemes of split rank one, or equivalently when the respective buildings are trees, their arithmetic subgroups can be studied by describing the associated quotient graphs and vertex stabilizers. This approach is called Bass-Serre theory in \cite[Chapter I, \S 5]{S}. 
In this book, Serre studies the quotient graph $\overline{X}=\mathrm{SL}_2(A) \backslash X$  defined from the action of $\mathrm{SL}_2(A)$ in the Bruhat-Tits tree $X=X(\mathrm{SL}_2(k),P)$. More specifically, he proves that:

\begin{theorem}\cite[Chapter II, \S 2.3]{S}\label{theo serre quot}
There exist a bounded connected subgraph $Y\subseteq \overline{X}$ and geodesic rays $\mathfrak{r}(\sigma)\subseteq \overline{X}$, for each $\sigma\in \mathrm{Pic}(A)$, such that
$\displaystyle \overline X=Y\cup \bigsqcup_{\sigma\in \mathrm{Pic}(A)} \mathfrak{r}(\sigma)$.
Moreover, $Y$ is finite, whenever $\mathbb{F}$ is finite.
\end{theorem}

Then, by using the Bass-Serre theory, Serre also presents $\mathrm{SL}_2(A)$ in term of amalgamated free products of certain simpler subgroups.

When $\mathbb{F}$ is a finite field, an interesting application of Theorem~\ref{theo serre quot}, which Serre introduces in \cite[Chapter II, \S 2.9]{S}, is the homological interpretation of the Euler-Poincar\'e characteristic of certain subgroups of $\mathrm{SL}_2(A)$.
In order to detail this, let us introduce the following definitions. 
We use $\sharp S$ to denote the cardinality of a set $S$.
In analogy with \cite[Prop. 14]{S2}, Serre defines the Euler-Poincar\'e characteristic $\chi(G)$, for any finite index subgroup $G$ of $\mathrm{SL}_2(A)$, by
\begin{equation}\label{eq euler-poincare char}
\chi(G)=  \sum_{v}  1/\sharp \mathrm{Stab}_G(v)- \sum_{e}  1/ \sharp \mathrm{Stab}_G(e),
\end{equation}
where $v$ (resp. $e$) runs through a system of representatives of the $G$-orbit of vertices (resp. geometrical edges) in $X$.
Both series in Equation \eqref{eq euler-poincare char} converge (c.f.~\cite[Ch. II, \S 2.3, Exercise 2]{S}). 

Special values of the zeta function $\zeta_{\mathcal{C}}$ defined on $\mathcal{C}$ are related to the Euler-Poincar\'e characteristic of $\mathrm{SL}_2(A)$. More specifically $\chi(\mathrm{SL}_2(A))=-(q^{\deg(P)}-1) \zeta_{\mathcal{C}}(-1)$ and $\chi(G)=[\mathrm{SL}_2(A):G] \chi(\mathrm{SL}_2)$, if $G$ has finite index in $\mathrm{SL}_2(A)$ (c.f.~\cite[Chapter II, \S 2.3, Exercise 2]{S}).
More general results are given in \cite[footnote in Pag. 158]{S2} and \cite{HN} for simply connected split simple algebraic group-schemes.

Now, let $H$ be a group and let $\mathfrak{H}=\lbrace H_{\sigma} \rbrace_{\sigma\in \Sigma}$ be an arbitrary non-empty family of subgroups of $H$. We define the $G$-module $R$ as the kernel of the augmentation $\mathbb{Z}$-morphism $\varepsilon:  \coprod_{\sigma \in \Sigma} \mathbb{Z}[H/H_{\sigma}] \to \mathbb{Z}$, defined by $\varepsilon(\delta)=1$, for all $\delta \in H/H_{\sigma}$ and all $\sigma \in \Sigma$. Then, given a $H$-module $M$, we define the $M$-valued relative homology groups of $H$ modulo $\mathfrak{H}$ by
\begin{equation}
H_i(H \text{ mod } H_{\sigma}, M):= \mathrm{Tor}_{i-1}^{\mathbb{Z}[H]}(R,M)= H_{i-1}(H, R \otimes M).
\end{equation}
The relative homology groups of $H$ modulo $\mathfrak{H}$ measure how far is $H_i(H,M)$, for $i > 0$, from being the co-product $\coprod_{\sigma\in \Sigma} H_i(H_{\sigma},M)$. 
More explicitly, if $H_i(H \text{ mod } H_{\sigma}, M)=H_{i+1}(H \text{ mod } H_{\sigma}, M)=\lbrace 0 \rbrace$, then $H_i(H,M) \cong \coprod_{\sigma\in \Sigma} H_i(H_{\sigma},M)$.

We say that a group $G$ is without $p'$-torsion when the elements of finite order in $G$ have $p$-power orders. It follows from \cite[Lemma 3.3]{MasonSury} that, when $\mathbb{F}=\mathbb{F}_{p^r}$, principal congruence subgroups of $\mathrm{SL}_2(A)$ are always without $p'$-torsion.

\begin{theorem}\cite[Chapter II, \S 2.9, Pag. 135]{S}\label{teo serre maximal unipotent subgroups}
Assume that $\mathbb{F}=\mathbb{F}_{p^r}$. Let $G$ be a finite index subgroup of $\mathrm{SL}_2(A)$ without $p'$-torsion. 
Let $\lbrace \xi_{\sigma} \rbrace_{\sigma \in \Sigma_0}$ be a set that represents the $G$-orbits in the visual limit of $X$, i.e. representatives of $G \backslash \mathbb{P}^1(k)$, and write $G_{\sigma}=\mathrm{Stab}_G(\xi_{\sigma})$.
Then, the set $\mathfrak{G}=\lbrace G_{\sigma} \rbrace_{\sigma \in \Sigma_0}$ is a representative system for the $G$-conjugacy classes of maximal unipotent subgroups in $G$. 
\end{theorem}

The following results describe the relative homology groups of $G$ modulo $\mathfrak{G}$ in terms of the Euler-Poincar\'e characteristic of $G$.

\begin{theorem}\cite[Chapter II, Theorem 14]{S}\label{teo serre Hom}
Assume that $\mathbb{F}=\mathbb{F}_{p^r}$.
Let $G$ be a finite index subgroup of $\mathrm{SL}_2(A)$ witout $p'$-torsion. Then, for all $i \neq 1$, the homology groups $H_i (G  \,\, \mathrm{mod}  \,\, G_{\sigma}, M)$ are trivial.
Moreover, if $M$ is $\mathbb{Z}$-finitely generated, then $H_1(G \,\, \mathrm{ mod } \,\, G_{\sigma}, M)\cong M^{-\chi(G)}$.
\end{theorem}

The purpose of the present work is to extend the preceding result to the context where $\mathrm{SL}_2$ is replaced by the special unitary group $\mathrm{SU}(3)=\mathrm{SU}(h)$ defined from a three-dimensional hermitian form $h$. 
This can be seen as the natural extension of Serre's results to the non-split context, since $\mathrm{SU}(3)$ is the only quasi-split non-split simply connected semisimple group of split rank $1$ (cf.~\cite[4.1.4]{BruhatTits2}). 
Moreover, this shows that Theorem~\ref{teo serre Hom} generalizes to any quasi-split simply connected semisimple group of split rank $1$.
We believe that a more general result along this line can exist for quasi-split groups of higher split rank and expect that this paper inspires.

\section{Context and main results}\label{section main results}
Let $\mathbb{F}$ be a field whose characteristic is not $2$. Let $\mathcal{C}$ be a smooth, projective, geometrically integral curve over $\mathbb{F}$, and also let $k$ the function field of $\mathcal{C}$. Let $\ell$ be a quadratic extension of $k$ such that $\mathbb{F}$ is algebraically closed in $\ell$, or equivalently, set a $2 : 1$ morphism $\psi: \mathcal{D} \to  \mathcal{C}$ of geometrically integral curves, and let us define $\ell$ as the function field of $\mathcal{D}$. For each element $x\in \ell$, we denote by $\overline{x}$ its image by the unique non-trivial element of $\mathrm{Gal}(\ell/k)$.

Let us introduce the following $\mathcal{C}$-group-scheme. Consider an affine subset of $\mathcal{C}$, which we write as $\mathrm{Spec}(R)$, where $R \subset k$ satisfies $\mathrm{Quot}(R)=k$. Let $S \subset \ell$ be the integral closure of $R$ in $\ell$. Then, by definition, $\mathrm{Spec}(S)$ is the fiber of $\psi$ over $\mathrm{Spec}(R)$. Let us denote by $\mathrm{SU}(h_R)$ the special unitary group-scheme defined from the $R$-hermitian form $h_R: S^3 \to R$ given by
\begin{equation}\label{eq h}
h_R(x_{-1}, x_0, x_1)= x_{-1} \overline{x_1}+ x_0 \overline{x_0}+ x_1 \overline{x_{-1}}.
\end{equation} 
Since we can cover $\mathcal{C}$ by affine subsets $\mathrm{Spec}(R_i)$ with affine intersection, the groups $\mathrm{SU}(h_{R_i})$ can be glue in order to define the $\mathcal{C}$-group-scheme $\mathcal{G}=\mathrm{SU}(h)$, which we call the special unitary group-scheme of $\mathcal{D}$ on $\mathcal{C}$.
 
 Let $P$ be a closed point on $\mathcal{C}$.
 Let us denote by $A$ the ring of function that are regular outside $P$.
 In all that concerns this work, we assume that $\ell/k$ is non-split at $P$, i.e. we assume that there exists a unique point $Q$ of $\mathcal{D}$ above $P$. 
 Then, the integral closure $B$ of $A$ in $L$ satisfies $\mathrm{Spec}(B)= \mathcal{D} \smallsetminus \lbrace Q \rbrace$. 
We denote by $\Gamma=\mathcal{G}(A)$ the group of $A$-point of $\mathcal{G}$. 
 By our assumption on $\ell$ and $P$, $\mathcal{G}$ is a quasi-split non-split simply connected semisimple group of split rank $1$. Then, the (rational) Bruhat-Tits building defined from $\mathcal{G}$ at $P$ is a tree, which we denote by $X = X(\mathcal{G},k,P)$. See \cite[\S 7.4]{BruhatTits1} for more details. The group $\mathcal{G}(k)$ acts on $X$ via simplicial maps.

Let $N=N_{\ell/k}$ and $T=\mathrm{Tr}_{\ell/k}$ be the norm and trace map of the quadratic extension $\ell/k$. In order to describe the visual limit $\partial_{\infty}(X)$ of $X$, i.e. its classes of rays modulo parallelism, it is convenient to introduce the following $k$-varieties:
 $$H(\ell,k):= \lbrace (u,v) \in R_{\ell/k}(\mathbb{G}_{a, \ell})^2: N(u)+ T(v)=0 \rbrace ,$$
 $$H(\ell,k)^{0}:= \lbrace (0,v) \in 0 \times R_{\ell/k}(\mathbb{G}_{a, \ell}): T(v)=0 \rbrace .$$
 Indeed, there exists a bijective map $f: H(\ell,k) \cup \lbrace \infty \rbrace \to \partial_{\infty}(X)$. See Corollary~\ref{coro visual limit and H(L,K)} for more details.

The main goal of this work is to extend Theorem~\ref{teo serre Hom} to arithmetic subgroups of the special unitary group. To do so, we study the action of subgroups $G$ of $\Gamma$ on the visual limit of $X$, or equivalently on $H(\ell,k) \cup \lbrace \infty \rbrace$. More specifically, we start by describing the Steinberg $G$-module $\mathrm{St}$ defined as the kernel of the \textit{augmentation} $\mathbb{Z}$-morphism $\epsilon: \mathbb{Z}[H(\ell,k) \cup \lbrace \infty \rbrace] \to \mathbb{Z}$ given by $\epsilon(z)=1$, for all $z \in H(\ell,k) \cup \lbrace \infty \rbrace$. 

\begin{theorem}\label{Main teo steinberg} Assume that $\mathbb{F}=\mathbb{F}_{p^r}$. Let $G$ be a finite index subgroup of $\Gamma$ without $p'$-torsion.
Then, the Steinberg module $\mathrm{St}$ is a finitely generated projective $\mathbb{Z}[G]$-module, whose stably free rank is $-\chi(G)$, i.e. there exist two finitely generated free $\mathbb{Z}[G]$-modules $L_0$ and $L_1$, whose rank are $l_0$ and $l_1$ respectively, such that
$$ \mathrm{St} \oplus L_0 \cong L_1, \text{ and } \chi(G)=l_0-l_1. $$
\end{theorem}

As in the case of $\mathrm{SL}_2$, the hypothesis of $G$ without $p'$-torsion holds for arithmetically interesting families of subgroups of $\Gamma$, as principal congruence subgroups, according to \cite[Lemma 3.3]{MasonSury}.

In a second step, and following Serre's ideas, we describe, in terms of the action of $G$ on $\partial_{\infty}(X) \cong H(\ell,k) \cup \lbrace \infty \rbrace$, a representative system for the $G$-conjugacy classes of maximal unipotent subgroups (c.f. \S\ref{Section maximal unipotent}).
This is valid even when $G$ has $p'$-torsion or $\mathbb{F}$ is infinite perfect field.

\begin{theorem}\label{main teo unipotent subgroups}
Assume that $\mathbb{F}$ is perfect.
Let $G$ be a finite index subgroup of $\Gamma$.
Let $\lbrace \xi_{\sigma} \rbrace_{\sigma \in \Sigma_0}$ be a set that represents the $G$-orbits in $\partial_{\infty}(X)$. 
For each $\sigma \in \Sigma_0$, let $(u,v)=f^{-1}(\xi_{\sigma}) \in H(\ell,k) \cup \lbrace \infty \rbrace$, and let
$$g_{u,v}= \small
    \begin{pmatrix}
        0 & 0& -1\\
        0 & -1 & -u\\
        -1 & \bar u & -v
    \end{pmatrix} \normalsize, \quad \quad g_{\infty}=\mathrm{id}.$$

Let $\mathcal{B}$ be the Borel subgroup of $\mathcal{G}$ defined as the scheme the upper-tringular matrices in $\mathcal{G}$ and let $\mathcal{U}_a$ be its unipotent radical.

For each $\sigma \in \Sigma_0$, let us denote by $G_{\sigma}$ the stabilizer in $G$ of $\xi_{\sigma}$, which we write as
$$G_{\sigma}= \mathrm{Stab}_{\mathcal{G}(k)}(\xi_{\sigma}) \cap G = g_{u,v}^{-1} \mathcal{B}(k) g_{u,v} \cap G. $$

For each $\sigma \in \Sigma_0$, let us denote $$ \mathcal{U}(G_{\sigma}) :=g_{u,v}^{-1} \mathcal{U}_a(k) g_{u,v} \cap G.$$

Then:
\begin{itemize}
\item[(1)] $\mathcal{U}(G_{\sigma})$ is the subgroup containing all the unipotent elements in $G_{\sigma}$ and $G_{\sigma}/\mathcal{U}(G_{\sigma})$ is isomorphic to a subgroup of $\mathbb{F}^{*}$,
\item[(2)] $\mathfrak{U}=\lbrace \mathcal{U}(G_{\sigma}) : \sigma \in \Sigma_0 \rbrace$ is a representative set of the conjugacy classes of maximal unipotent subgroups of $G$, and
\item[(3)] if $\mathbb{F}=\mathbb{F}_{p^r}$ and $G$ is without $p'$-torsion, then $\mathcal{U}(G_{\sigma})=G_{\sigma}$, for all $\sigma \in \Sigma_0$.
\end{itemize} 
\end{theorem}


The preceding result has some interesting consequences on the number of conjugacy classes of maximal unipotent subgroups in $G$ (c.f. \S\ref{subsection consequences of main teo max unip sub}).

Then, by heavily using Theorem~\ref{Main teo steinberg} we describe the homology of $G$ modulo $\mathfrak{U}$, as defined in Theorem~\ref{main teo unipotent subgroups}, extending Theorem~\ref{teo serre Hom} to the special unitary group.

\begin{theorem}\label{main teo Hom}
Assume that $\mathbb{F}=\mathbb{F}_{p^r}$. Let $G$ be a finite index subgroup of $\Gamma$ without $p'$-torsion. Then, for all $i \neq 1$, the homology groups $H_i (G \,\,\mathrm{ mod}\,\, G_{\sigma}, M)$ are trivial.
Moreover, if $M$ is $\mathbb{Z}$-finitely generated, then $H_1(G \,\,\mathrm{mod}\,\, G_{\sigma}, M) \cong M^{-\chi(G)}$.
In particular, for each $i \in \mathbb{Z}_{>1} $ we have $H_{i}(G,M) \cong \coprod_{\sigma \in \Sigma_0} H_i(G_{\sigma},M)$, and, for $i=1$, the following sequence is exact:
$$
0 \to \coprod_{\sigma \in \Sigma_0} H_1(G_{\sigma},M) \to H_1 (G,M) \to H_1(G \text{ mod } G_{\sigma},M).$$
\end{theorem}

As an application of Theorem~\ref{main teo Hom}, we prove in Proposition~\ref{prop abelianization of G} that the torsion group $\mathrm{Tor}\left(G^{\mathrm{ab}}\right)$ is naturally isomorphic to a finite index subgroup of a product of finitely many $B$-fractional ideals, and that the rank of the free part of $G^{\mathrm{ab}}$ is bounded by $-\chi(G)$.
In particular, we show that, even though $\Gamma^{\mathrm{ab}}$ is finitely generated (c.f.~\cite[Prop. 9.7]{Arenas}), the group $G^{\mathrm{ab}}$ is not finitely generated. Finally, we describe the torsion subgroup of the abelianization of principal congruence subgroups in $\Gamma$.

\section{Conventions and preliminaries}\label{section conv and prel}

\subsection{Convention on graphs}\label{subsection graphs}
A graph $\mathfrak{g}$ consists of two sets $\mathrm{V}=\mathrm{V}(\mathfrak{g})$ and $\mathrm{E}=\mathrm{E}(\mathfrak{g})$, and three maps
$$ o: \mathrm{E} \to  \mathrm{V}, \, \, e \mapsto o(e), \quad \quad t: \mathrm{E} \to  \mathrm{V}, \, \,  e \mapsto  t(e), \quad \quad r: \mathrm{E} \to \mathrm{E}, \, \,  e \mapsto \overline{e}, $$
which satisfy the following condition: for each $e \in \mathrm{E}$ we have $r(r(e))=e$, $r(e) \neq e$ and $o(e)=t(r(e))$.
The sets $\mathrm{V}$ and $\mathrm{E}$ are called vertex set and edge set, respectively, and the functions $o,t$ and $r$ are called origin, terminus and reverse, respectively. The vertex $o(y)=t(r(y))$ is called the origin of $y$, and the vertex $t(y)=o(r(y))$ is called the terminus of $y$. A geometrical edge of $\mathfrak{g}$ is a set of the form $\lbrace e, r(e) \rbrace$, for some $e \in \mathrm{E}$. The set of geometrical edges of $\mathfrak{g}$ is denoted by $\mathrm{geom(E)}(\mathfrak{g})$. An orientation of a graph $\mathfrak{g}$ is a subset $\mathrm{E}^{+}$ of $\mathrm{E}$ such that $\mathrm{E}$ is the disjoint union of $\mathrm{E}^{+}$ and $r(\mathrm{E}^{+})$. It is not hard to see that orientations always exist, and that, once an orientation is fixed, then there exists a bijection between $\mathrm{E}^{+}$ and $\mathrm{geom(E)}(\mathfrak{g})$.

A morphism $f: \mathfrak{g} \to  \mathfrak{g}'$ between two graphs is a pair son functions $f_{\mathrm{V}}: \mathrm{V}(\mathfrak{g}) \to \mathrm{V}(\mathfrak{g}')$ and $f_{\mathrm{E}}: \mathrm{E}(\mathfrak{g}) \to \mathrm{E}(\mathfrak{g}')$ preserving the origin, terminus and reverse functions. A similar convention applies for group actions. We say that an action of a group $H$ on a graph $\mathfrak{g}$ does not have invertions when it preserves orientations, or equivalently, when $h \cdot y \neq r(y)$, for every $y \in \mathrm{E}(\mathfrak{g})$, and every $h \in H$. If $H$ acts on $\mathfrak{g}$ without inversions, then we can define the quotient graph $H \backslash \mathfrak{g}$ in the obvious way; the vertex set (edge set, resp.) of $H \backslash \mathfrak{g}$ is the quotient $H \backslash \mathrm{V}(\mathfrak{g})$ ($H \backslash \mathrm{V}(\mathfrak{g})$, resp.). This type of quotients are used in Bass-Serre theory to study the structure of groups acting on trees, as we mention in \S\ref{section introduction}.


Let $n$ be a positive integer. Let us define the graph $\mathfrak{path}_n$ by setting $\mathrm{V}=\lbrace 0, 1, \cdots, n \rbrace$ and $\mathrm{E}= \lbrace a_i, \overline{a_i}: 0 \leq i < n \rbrace$, where $r(a_i)=\overline{a_i}$, $o(a_i)=i$ and $t(a_i)=i+1$. Then, we define a path of length $n$ in a graph $\mathfrak{g}$ as an injective morphism $\mathfrak{l}: \mathfrak{path}_n \to \mathfrak{g}$. If $v_i=\mathfrak{l}(i)$, for all $ 0 \leq i \leq n $, then we say that $\mathfrak{l}$ is a path from $v_0$ to $v_n$. A graph $\mathfrak{g}$ is connected when there exists a path $\mathfrak{l}=\mathfrak{l}(v,w)$ connecting any two vertices $v,w \in \mathrm{V}(\mathfrak{g})$.
We define the graph $\mathfrak{circ}_n$ by setting $\mathrm{V}=\mathbb{Z}/n\mathbb{Z}$ and $\mathrm{E}= \lbrace e_i, \overline{e_i}: i \in \mathbb{Z}/n\mathbb{Z} \rbrace$, where $r(e_i)=\overline{e_i}$, $o(e_i)=i$ and $t(e_i)=i+1$. A loop of length $n$ in a graph $\mathfrak{g}$ is an injective morphism $\mathfrak{circ}_n \to \mathfrak{g}$. A tree is a connected graph without loops.

Let $\mathfrak{g}$ be a graph. We define a ray $\mathfrak{r}$ of $\mathfrak{g}$ as a subcomplex whose vertex set $\mathrm{V}(\mathfrak{r})=\lbrace v_i \rbrace_{i=0}^{\infty}$ and edge set $\mathrm{E}(\mathfrak{r})=\lbrace e_i \rbrace_{i=0}^{\infty}$ satisfies the identities $o(e_i)=v_i$, $t(e_i)=v_{i+1}$, and $v_i \neq v_j$, for all pair of different indices $(i,j)$.
The vertex $v_0$ is called the initial vertex of $\mathfrak{r}$.
Note that, the set of rays of a given graph can be empty, for instance, when the graph has a finite number of vertices.
In all that follows we just work with graphs that contain at least one ray.
In order to define the visual limit of a ray, let us take two rays $\mathfrak{r}_1$ and $\mathfrak{r}_2$, and let us write $\mathrm{V}(\mathfrak{r}_1)=\lbrace v_i \rbrace_{i=0}^{\infty}$ and $\mathrm{V}(\mathfrak{r}_2)=\lbrace w_i \rbrace_{i=0}^{\infty}$.
Then, we say that $\mathfrak{r}_1$ and $\mathfrak{r}_2$ belongs to the same visual limit if there exists $i_0,s \in \mathbb{Z}_{\geq 0}$ such that $v_i=w_{i+s}$, for all $i \geq i_0$.
This defines a equivalence relation on the set of rays of $\mathfrak{g}$.
We denote by $\partial_{\infty}(\mathfrak{r})$ the class, called the visual limit, of the ray $\mathfrak{r}$.
We denote by $\partial_{\infty}(\mathfrak{g})$ the set of all classes in $\mathfrak{g}$.


\subsection{On the radical datum of the special unitary group}\label{subsection radical datum}
Here we recall some well known facts about the $\mathcal{C}$-group scheme $\mathcal{G}$. For more detailed definitions, see, for instance, \cite[4.1]{BruhatTits2}, or \cite[\S 4 Case 2. p. 43-50]{Landvogt}. We start by recalling that $\mathcal{G}$ embeds into $\mathrm{SL}_{3, \mathcal{C}}$ (c.f. \cite[\S 3.2]{Arenas}). By using this representation, we identify $\mathcal{G}$ with its image, and the elements in $\mathcal{G}(k)$ (resp. $\mathcal{G}(A)$) with $3 \times 3$ matrices with coefficients in $\ell$ (resp. $B$). 

Given a field $F$, let $\mathbb{G}_{m,F}$ be the multiplicative group defined over $F$. Let us write:
$$ \mathrm{diag}(x,y,z):= \small
    \begin{pmatrix}
        x & 0& 0\\
        0 & y & 0\\
        0 & 0 & z
    \end{pmatrix} \normalsize.$$
A $k$-split maximal torus $\mathcal{S}$ of $\mathcal{G}$ consists in the subgroup of diagonal matrices of the from $\mathrm{diag}(s,1,s^{-1})$, where $s \in \mathbb{G}_{m,k}$. The centralizer $\mathcal{T}=\mathcal{Z}_{\mathcal{G}}(\mathcal{S})$ of $\mathcal{S}$ is a $k$-maximal torus, and it admits a parametrization of the form $\tilde{a}: \mathrm{R}_{\ell/k}(\mathbb{G}_{m,\ell}) \to \mathcal{T}$ with $\tilde{a}(t)=\mathrm{diag}(t, \overline{t}t^{-1}, \overline{t}^{-1})$. This torus $\mathcal{T}$ splits over $\ell$ into the dimension $2$ torus $\mathcal{T}_{\ell} \cong \mathbb{G}_{m,\ell}^2$.

A basis of characters of the $\ell$-split torus $\mathcal{T}_{\ell}$ consists in characters $\lbrace \alpha, \overline{\alpha} \rbrace$ defined by $\alpha (\mathrm{diag}(x,y,z))
		= y z^{-1}$, and $ \overline{\alpha}(\mathrm{diag}(x,y,z))= x y^{-1}.$
We denote by $a$ (resp. $2a$) the restriction of $\alpha$ (resp. $\alpha+\overline{\alpha}$) to $\mathcal{S}$. Then, the $k$-root system of $\mathcal{G}$ is $\Phi=\lbrace \pm a, \pm 2a \rbrace$. We denote by $a^{\vee}$ (resp. $(2a)^{\vee}$) the coroot of $a$ (resp. $2a$). It is well-known that $a$ generates the $\mathbb{Z}$-module of characters of $\mathcal{S}$, while $(2a)^{\vee}$ generates its $\mathbb{Z}$-module of cocharacters. 

We fix $\lbrace a, 2a \rbrace$ as a set of positive roots in $\Phi$. This election defines a $k$-Borel subgroup $\mathcal{B}$ of $\mathcal{G}$ containing $\mathcal{T}$. Indeed, this group $\mathcal{B}$ corresponds to the set of upper-triangular matrices in $\mathcal{G}$. 
We parametrize its unipotent root subgroups $\mathcal{U}_a$ and $\mathcal{U}_{2a}$ by
$$\begin{array}{cccc}
    \mathrm{u}_a:& H(\ell,k) &\to  & \mathcal{U}_a\\
     & (u,v) &\mapsto & \small \begin{pmatrix} 1 & -\bar{u} & v \\ 0 & 1 & u \\ 0 & 0 & 1\end{pmatrix} \normalsize
    \end{array}
    \, \, \text{,  and } \, \, 
    \begin{array}{cccc}
    \mathrm{u}_{2a}:& H(\ell,k)^0 &\to  & \mathcal{U}_{2a}\\
     & v &\mapsto & \small \begin{pmatrix} 1 & 0 & v \\ 0 & 1 & 0 \\ 0 & 0 & 1\end{pmatrix} \normalsize
    \end{array}.$$
Then, we have $\mathcal{B}=\left\lbrace \tilde{a}(t) \mathrm{u}_a(x,y): t \in \ell^{*}, \, (x,y) \in H(\ell,k) \right\rbrace.$
Each unipotent subgroups of $\mathcal{G}$ corresponding to a negative roots $b \in \lbrace -a, -2a \rbrace$ is parametrized by the homomorphism $\mathrm{u}_{b} :=\mathrm{s} \cdot \mathrm{u}_{-b} \cdot \mathrm{s}$, where $\mathrm{s}$ is the $3 \times 3$ matrix with $-1$ in the anti-diagonal and $0$'s in any other coordinate.

\subsection{The Bruhat-Tits tree}\label{subsection the BTT}
Here we define the (rational) Bruhat-Tits tree $X=X(\mathcal{G},k,P)$ defined from $\mathcal{G}$ and the discrete valuation map $\nu=\nu_P$ of $k$ defined from $P$. By abuse of notation, also denote by $\nu$ the valuation on $\ell$ induced by $Q$, and assume that $\nu(\ell^{*})= \mathbb{Z}/e_P \mathbb{Z}$, where $e_P$ is the ramification index of $P$ in $\ell$.
The standard apartment of $X$ defined from $\mathcal{S}$ is, by definition, $\mathbb{A}_0=X_{*}(\mathcal{S}) \otimes_{\mathbb{Z}} \mathbb{R} \cong \mathbb{R} a^{\vee}$. The vertex set of $\mathbb{A}_0$ corresponds to the elements $x=r a^{\vee}$ such that $a(x)=2r$ belongs to $\Gamma_a:= \frac{1}{2} \nu(\ell^{*})$. See \cite[4.2.21(4) and 4.2.22]{BruhatTits2} for more details. It follows from \cite[6.2.10]{BruhatTits1} and \cite[4.2.7]{BruhatTits2} that the group of $k$-points $\mathcal{N}_{\mathcal{G}}(\mathcal{S})(k)$ of the normalizer $\mathcal{N}_{\mathcal{G}}(\mathcal{S})$ acts on $\mathbb{A}_0$ via:
$$ \tilde{a}(t) \cdot x = x- \frac{1}{2}\nu(t)a^{\vee}, \text{ and } \mathrm{s} \cdot x=-x, \quad  \forall x \in\mathbb{A}_0, \, \,  \, \forall t \in \mathrm{R}_{\ell/k}(\mathbb{G}_{m,\ell}).$$

For each $b\in \lbrace a,-a \rbrace$, let us define $\mathcal{U}_{b,x}(k)$ as the group $\lbrace \mathrm{u}_{b}(x,y) : (x,y) \in H(\ell,k), \nu(y) \geq - 2b(x) \rbrace.$
Then, the rational Bruhat-Tits tree $X$ defined from $\mathcal{G}(k)$ is the graph defined as the gluing of multiple copies of $\mathbb{A}_0$:
\begin{equation}\label{eq SU tree}
X=X(\mathcal{G},k,P):= \mathcal{G}(k) \times \mathbb{A}_0/\backsim,
\end{equation}
with respect to the equivalence relation: $(g,x) \backsim (h,y)$ if and only if there is $n \in \mathcal{N}_{\mathcal{G}}(\mathcal{S})(k)$ such that $y=n \cdot x$ and $g^{-1}hn \in \mathcal{U}_{x}(k)$, where
$\mathcal{U}_{x}(k)= \left\langle \mathcal{U}_{a, x}(k), \mathcal{U}_{-a, x}(k) \right\rangle$.

\subsection{The visual limit of the Bruhat-Tits tree}\label{subsection action of gamma}

Recall that $\mathcal{G}(k)$ acts on $X$. This action extends to the visual limit of $X$ via $g \cdot \partial_{\infty}(\mathfrak{r}_1)=  \partial_{\infty}(g \cdot \mathfrak{r}_1)$, for all $g \in \mathcal{G}(k)$, since $\partial_{\infty}(\mathfrak{r}_1)=\partial_{\infty}(\mathfrak{r}_2)$ implies $g \cdot \partial_{\infty}(\mathfrak{r}_1)=g \cdot \partial_{\infty}(\mathfrak{r}_2)$.

\begin{lemma}\label{trans action}
$\mathcal{G}(k)$ acts transitively on $\partial_{\infty}(X)$.
\end{lemma}

\begin{proof}
Let $\mathfrak{r}_1$ and $\mathfrak{r}_2$ two rays in $X$. Let $\mathbb{A}_1$ and $\mathbb{A}_2$ two apartments containing $\mathfrak{r}_1$ and $\mathfrak{r}_2$, respectively. Since $\mathcal{G}(k)$ acts transitively on the apartments of $X$, there exists $g \in \mathcal{G}(k)$ satisfying $g \cdot \mathbb{A}_1=\mathbb{A}_2$. Moreover, by multiplying $\mathbb{A}_2$ for a suitable element in $\mathcal{G}(k)$, we can assume that $\mathbb{A}_2$ is the standard apartment $\mathbb{A}_0 \subset X$. Then, $g \cdot \mathfrak{r}_1$ and $\mathfrak{r}_2$ are both rays in $\mathbb{A}_0$. Since $\lbrace \mathrm{id}, \mathrm{s} \rbrace$ acts transitively on the visual limits of $\mathbb{A}_0$, we conclude that there exist $\epsilon \in \lbrace 0,1 \rbrace$ such that $\mathrm{s}^{\epsilon}g \cdot \partial_{\infty}(\mathfrak{r}_1)=\partial_{\infty}(\mathfrak{r}_2)$, whence the result follows.
\end{proof}

Let us denote by $\mathfrak{r}(\infty)$ the ray defined as $\lbrace x \in \mathbb{A}_0 : a(x) \geqslant 0 \rbrace$, and denote by $\xi_{\infty}$ its visual limit. The following is a classical result in Building theory.

\begin{lemma}\cite[§6.2.4]{Brown}\label{stab visual limit}
The stabilizer of $\xi_{\infty}$ in $\mathcal{G}(k)$ is $\mathcal{B}(k)$.
\end{lemma}

For each $(u,v)\in H(\ell,k)$, we define
$$g_{u,v} := \mathrm{u}_{-a}(u,v) \cdot \mathrm{s} = \mathrm{s} \cdot \mathrm{u}_{a}(u,v)= \small
    \begin{pmatrix}
        0 & 0& -1\\
        0 & -1 & -u\\
        -1 & \bar u & -v
    \end{pmatrix} \normalsize \text{ and } \xi_{(u,v)}:= g_{u,v}^{-1} \cdot \xi_{\infty}.
$$
In order to homogenize the expression of visual limits, in all what follows we adopt the following convention: for $\xi=\xi_{\infty}$ we write $(u,v)=\infty$ and we put $g_{\infty}=\mathrm{id}$.

\begin{corollary}\label{coro visual limit and H(L,K)}
The map $f:  H(\ell,k) \cup \lbrace \infty \rbrace \to \partial_{\infty}(X)$ defined as $f((u,v))= \xi_{(u,v)}$, where $(u,v) \in H(\ell,k) \cup \lbrace \infty \rbrace $, is bijective. In particular, there exists a bijection between $\mathcal{G}(k)/\mathcal{B}(k)$ and $H(\ell,k) \cup \lbrace \infty \rbrace$. 
\end{corollary}

\begin{proof}
Firstly, we prove that $f$ is injective. Indeed, if $g_{u_1,v_1}^{-1} \cdot \xi_{\infty}= g_{u_2,v_2}^{-1} \cdot \xi_{\infty}$, then $g_{u_1,v_1}g_{u_2,v_2}^{-1} \in \mathcal{B}(k)$, as Lemma~\ref{stab visual limit} shows. In particular, if $(u_1,v_1)=\infty$, then $(u_1,v_2) =\infty$, and inversely. In any other case, we have that $g_{u_1,v_1}g_{u_2,v_2}^{-1}=\mathrm{s} \mathrm{u}_a(u_1,v_1) \mathrm{u}_a(u_2,v_2)^{-1} \mathrm{s}^{-1}$ belongs to $\mathrm{s} \mathcal{B}(k) \mathrm{s}^{-1} \cap \mathcal{B}(k)$. 
Then, $\mathrm{u}_a(u_1,v_1)= \mathrm{u}_a(u_2,v_2)$, since $\mathrm{s} \mathcal{B}(k) \mathrm{s}^{-1} \cap \mathcal{B}(k)=\mathcal{T}(k)$. We conclude that $(u_1,v_1)=(u_2,v_2)$ in all cases.

Now, we claim that $f$ is surjective. Indeed, let $\xi \in \partial_{\infty}(X)$ be an arbitrary element. It follows from Lemma~\ref{trans action} that there exists $g \in \mathcal{G}(k)$ such that $\xi= g \cdot \xi_{\infty}$.
Moreover, it follows from the Bruhat decomposition that $\mathcal{G}(k)= \mathcal{B}(k) \cup \mathcal{U}_a(k)\mathrm{s}\mathcal{B}(k)$, where $\mathcal{B}(k) \cap \mathcal{U}_a(k)\mathrm{s}\mathcal{B}(k)=\emptyset$.
Since $\mathcal{B}(k)=\mathrm{Stab}_{\mathcal{G}(k)}(\xi_{\infty})$ according to Lemma~\ref{stab visual limit}, we have $\xi= g \cdot \xi_{\infty}=\xi_{\infty}$, if $g \in \mathcal{B}(k)$, while $\xi= \mathrm{u}_a(x,y) \mathrm{s} \cdot \xi_{\infty}= (\mathrm{s} \,  \mathrm{u}_a(-x,\overline{y}))^{-1} \cdot \xi_{\infty}=\xi_{(-x,\overline{y})}$, in any other case.
This proves the claim. Thus, we conclude that $f$ is bijective.

Note that Lemma~\ref{trans action} and Lemma~\ref{stab visual limit} implies $\partial_{\infty}(X) \cong \mathcal{G}(k)/\mathcal{B}(k)$. This proves the second statement.
\end{proof}

\subsection{Actions on the Bruhat-Tits tree}\label{subsection action of arit groups}
 
Since $\mathcal{G}$ is a semi-simple group-scheme, there exists a bipartition of the Bruhat-Tits tree $X$ that is respected by the action of $\mathcal{G}(k)$. This implies that any subgroup of $\mathcal{G}(k)$ acts on $X$ without inversions. In particular, any subgroup of $\mathcal{G}(k)$ defines a quotient graph (c.f.~\S\ref{subsection graphs}). The following result describes the quotient graph defined by the action of $\Gamma=\mathcal{G}(A)$ on $X$.

\begin{theorem}\cite[Theorem 2.1 and Theorem 7.1]{Arenas}\label{Teo ABLL quotient}
There exist a connected subgraph $Y$ of $\Gamma \backslash X $ and geodesical rays $\mathfrak{r}(\sigma)\subseteq \Gamma \backslash X $ with initial vertex $v_\sigma$, for each $\sigma\in \Gamma \backslash \partial_{\infty}(X)$, such that $\mathrm{V}(Y)\cap \mathrm{V}(\mathfrak{r}(\sigma))=\{v_\sigma\}$, $\mathrm{E}(Y)\cap \mathrm{E}(\mathfrak{r}(\sigma))=\emptyset$, and
$$\Gamma \backslash X=Y\cup \bigsqcup_{\sigma\in \Gamma \backslash \partial_{\infty}(X)} \mathfrak{r}(\sigma),$$
where the set of orbits $\Gamma \backslash \partial_{\infty}(X)$ is in bijection with the Picard group $\mathrm{Pic}(B)$ of the Dedekind ring $B$. 
Moreover, $Y$ is finite when $\mathbb{F}$ is finite.
\end{theorem}

In particular, in all that follows, we write $\lbrace \mathfrak{r}(\sigma): \sigma \in \mathrm{Pic}(A) \rbrace$ to denote the set of geodesical rays defined in Theorem \ref{Teo ABLL quotient}.
The vertex stabilizers defined by the action of $\Gamma$ on $X$ can be characterized as follows.


\begin{theorem}\cite[Lemma 4.9 and Proposition 4.12]{Arenas}\label{Teo ABLL stab}
For each $\xi \in \partial_{\infty}(X)$ there exists a ray $\mathfrak{r}_\xi$, such that $\mathrm{V}(\mathfrak{r}_\xi)=\lbrace v_{i}(\xi) \rbrace_{i=0}^{\infty}$, where $v_{i}(\xi)$ and $v_{i+1}(\xi)$ are neighbors, and
\begin{itemize}
    \item[(a)] $\partial_{\infty}(\mathfrak{r}_\xi)=\xi$,
    \item[(b)] $\mathrm{Stab}_{\Gamma}(v_i(\xi)) $ is strictly contained in $\mathrm{Stab}_{\Gamma}(v_{i+1}(\xi))$, and
    \item[(c)] $\mathrm{Stab}_{\Gamma}(\xi)= \bigcup_{i=0}^{\infty} \mathrm{Stab}_{\Gamma}(v_i(\xi))$.
\end{itemize}
For each $\xi \in \partial_{\infty}(X)$, we set $(u,v) =f^{-1}(\xi)\in H(\ell,k) \cup \lbrace \infty \rbrace$.
Let us write $$\Gamma_{\xi} =g_{u,v}^{-1} \mathcal{B}(k) g_{u,v} \cap \Gamma, \quad U_{\xi} = g_{u,v}^{-1} \mathcal{U}_a(k) g_{u,v} \cap \Gamma,\quad U_\xi^{0}=g_{u,v}^{-1} \mathcal{U}_{2a}(k) g_{u,v} \cap \Gamma.$$
Then:
\begin{itemize}
\item[(d)] $\Gamma_\xi=\mathrm{Stab}_{\Gamma}(\xi)$,
\item[(e)] $U_\xi$ is a normal subgroup of $\Gamma_\xi$ and $\Gamma_\xi/U_\xi$ is isomorphic to a subgroup of $\mathbb{F}^{*}.$
\end{itemize}

Moreover, for $H(u,v):=\mathrm{u}_a^{-1} \left(g_{u,v} U_\xi g_{u,v}^{-1} \right)$ and $H(u,v)^{0}:=\mathrm{u}_a^{-1}\left(g_{u,v} U_\xi^{0} g_{u,v}^{-1} \right)$ we have:
\begin{itemize}
\item[(f)] $U_\xi \cong H(u,v)$ and $U_{\xi}^{0} \cong H(u,v)^{0}$, and
\item[(g)] $H(u,v)^{0}$ and $H(u,v)/H(u,v)^{0}$ are both isomorphic to non-trivial finitely generated $A$-modules of $\ell$ containing a non-zero ideal $I \subseteq B$.
\end{itemize}

\end{theorem}

Note that, for each geodesical ray $\mathfrak{r}(\sigma)$ in $\Gamma \backslash X$ there exists a subray $\mathfrak{r}'(\sigma) \subseteq \mathfrak{r}(\sigma)$ whose lifting in $X$ satisfies Theorem~\ref{Teo ABLL stab}.
Then, when $\mathbb{F}$ is finite, by attaching a finite graph to $Y$ if needed, we assume that each cusps ray $\mathfrak{r}(\sigma)$ has a lifting in $X$ satisfying Theorem~\ref{Teo ABLL stab}.

\begin{corollary}\label{coro finitud} Assume that $\mathbb{F}$ is finite. Then, given $N \in \mathbb{Z}_{>0}$, there exist finitely many vertices (resp. edges) $\overline{v}$ in $\Gamma \backslash X$ (resp. $\overline{e}$ in $\Gamma \backslash X$) such that, for every $v' \in \Gamma \cdot v $ (resp. $e' \in \Gamma \cdot v$), we have $\sharp \mathrm{Stab}_{\Gamma}(v') \leq N$ (resp. $\sharp \mathrm{Stab}_{\Gamma}(e') \leq N$).
\end{corollary}

\begin{proof}
Let $\lbrace\mathfrak{r}(\sigma) : \sigma \in \mathrm{Pic}(B) \rbrace$ as above. For each $\sigma \in \mathrm{Pic}(B)$, let $\mathfrak{r}(\sigma)^{\circ} \subseteq X$ be a lifting of $\mathfrak{r}(\sigma)$ satisfying Theorem~\ref{Teo ABLL stab}. Then, it follows from Theorem~\ref{Teo ABLL stab}(b) that, for each $\sigma \in \mathrm{Pic}(B)$, there exists a ray $\mathfrak{r}(\sigma)^{'} \subset \mathfrak{r}(\sigma)^{\circ}$ such that for all $v' \in \mathrm{V}(\mathfrak{r}(\sigma)^{'})$ (resp. $e' \in \mathrm{E}(\mathfrak{r}(\sigma)^{'})$) we have $\sharp \mathrm{Stab}_{\Gamma}(v') > N$ (resp. $\sharp \mathrm{Stab}_{\Gamma}(e') > N$). Thus, the result follows from the finiteness of the graph $Y$ defined in Theorem~\ref{Teo ABLL quotient}.
\end{proof}

\section{Fixed points in the visual limit}

The main goal of this section is to prove following result.

\begin{proposition}\label{Lemma action of a p-group} Assume that $\mathbb{F}=\mathbb{F}_{p^r}$.
Let $P$ be a non-trivial finite $p$-subgroup of $\Gamma=\mathcal{G}(A)$. Then, the natural action of $P$ on $\partial_{\infty}(X) \cong H(\ell,k) \cup \lbrace \infty \rbrace$ has exactly one fixed point.
\end{proposition}

We divide the proof of Proposition~\ref{Lemma action of a p-group} in two parts. The fist one is devoted to the existence of a fixed point and the later to its uniqueness.

\begin{lemma}\label{lemma corr with an lines}
There exists a one-to-one correspondence between $\partial_{\infty}(X)$ and the set of lines in $\ell^3$ that are isotropic with respect to the hermitian form defined by Equality \eqref{eq h}.
\end{lemma}

\begin{proof}
Recall that every isotropic line in $\ell^3$ is contained in a hyperbolic plane, which has an isometry 
switching its two isotropic lines. Then, Witt's theorem shows that $\mathcal{G}(k)$ acts transitively on the set of such lines.
Moreover, it is not hard to prove that the stabilizer of the line generated by $(1,0,0)$ is the group of upper triangular matrices in $\mathcal{G}(k)$, i.e.~$\mathcal{B}(k)$. Thus, the set of isotropic lines in $\ell^3$ is in natural correspondence with $\mathcal{G}(k)/\mathcal{B}(k)$. Then, the result follows from Corollary~\ref{coro visual limit and H(L,K)}.
\end{proof}

\begin{lemma}\label{lemma fix points in V}
Assume that $\mathbb{F}=\mathbb{F}_{p^r}$.
Let $P$ be a $p$-subgroup of $\mathcal{G}(k)$ and let $V$ be a $P$-stable $\mathbb{F}$-subspace of $\ell^3$. Then the exists a nonzero vector $w \in V$ that is fixed by $P$.
\end{lemma}

\begin{proof} Fix $v \in V \smallsetminus \lbrace 0 \rbrace $ and define $W$ as the $\mathbb{F}$-vector space generated by all elements in the $P$-orbit of $v$. Since $P$ is finite, the $P$-orbit of $v$ also finite.
Thus, $W$ has finite dimension.
In particular, since $\mathbb{F}=\mathbb{F}_{p^r}$, the cardinality of $W$ is $p^n$, for some $n \in \mathbb{Z}_{>0}.$
Moreover, since $\mathcal{G}(k) \subseteq \mathrm{SL}_3(\ell)$, the group $P$ acts on $V \subseteq \ell^3$ by linear transformations. Then, $P \cdot W \subseteq W$, which implies that $P$ acts on $W$. Since $0 \in W$ is a $P$-fixed point, it follows from an easy orbit count that the cardinality of $\lbrace x \in W : P \cdot x=x \rbrace$ is divisible by $p$. In particular, we conclude that there exists a nonzero $P$-fixed point in $W$.
\end{proof}

\begin{lemma}\label{lemma fixed point}
Assume that $\mathbb{F}=\mathbb{F}_{p^r}$.
Let $P$ be a non-trivial finite $p$-subgroup of $\mathcal{G}(k)$. Then $P$ fixes a point of $\partial_{\infty}(X)$.
\end{lemma}

\begin{proof}
According to Lemma~\ref{lemma corr with an lines}, we have to show that there exists a nonzero vector in $\ell^3$ that is fixed by $P$ and isotropic with respect the hermitian form defined in Eq. \eqref{eq h}.

Lemma~\ref{lemma fix points in V} applied to $V=\ell^3$ implies that there exists $w \in \ell^{3} \smallsetminus \lbrace 0 \rbrace$ which is fixed by $P$. If $w$ is isotropic, then the result follows. Assume that $w$ is anisotropic. Then, $\ell^3$ decomposes as the direct sum of two $\ell[P]$-stable modules $V_0=\langle w \rangle$ and $V_1=\langle w\rangle^{\perp}$. 
By applying Lemma~\ref{lemma fix points in V} on $V=V_1$, we show that there exists a nonzero vector $w_1 \in V_1$ that is fixed by $P$.
If $w_1$ is isotropic, then the proof is concluded. 
If $w_1$ is anisotropic, then $V_1$ decomposes in the sum of two
$\ell[P]$-stable lines.
In this case, $P$ acts on $\ell^3$ via a diagonal representation $\rho: P \to \mathrm{GL}_3(\ell)$. Set $g \in P \smallsetminus \lbrace e \rbrace$ and, as in \S\ref{subsection radical datum}, write $\rho(g)= \mathrm{diag}(\lambda_1, \lambda_2, \lambda_3)$, for some $\lambda_1=\lambda_1(g), \lambda_2=\lambda_2(g), \lambda_3=\lambda_3(g) \in \ell^{*}$. Since $g^{p^t}=e$, for some $t \in \mathbb{Z}_{\geq 0}$, we have $\mathrm{diag}(\lambda_1^{p^t}, \lambda_2^{p^t}, \lambda_3^{p^t})=\rho(g)^{p^t}=\mathrm{id}$. Thus, we get $\lambda_1=\lambda_2=\lambda_3=1$, since $\mathrm{char}(\ell)=p$. Hence, we deduce that $\rho$ is trivial. In other words, we conclude that any vector in $\ell$, and in particular any isotropic line, is fixed by $P$.
\end{proof}

\begin{proof}[Proof of Proposition~\ref{Lemma action of a p-group}]
Let $P \subset \Gamma=\mathcal{G}(A)$ be a non-trivial finite $p$-group.
It follows from Lemma~\ref{lemma fixed point} that $P$ has a fixed point in $\partial_{\infty}(X)$. Thus, is suffices to prove uniqueness.
To do so, assume that $P$ fixes two points $\xi_1, \xi_2 \in \partial_{\infty}(X)$. In other words, assume that\begin{equation}\label{eq P subset stab}
P \subseteq \mathrm{Stab}_{\mathcal{G}(k)}(\xi_1) \cap \mathrm{Stab}_{\mathcal{G}(k)}(\xi_2).
\end{equation}
Since $\mathcal{G}(k)$ acts transitively on $\partial_{\infty}(X)$ (c.f. Lemma~\ref{trans action}), we can write $\xi_1= g_1 \cdot \xi_{\infty}$ and $\xi_2= g_2 \cdot \xi_{\infty}$, for some $g_1, g_2 \in \mathcal{G}(k)$. By the Bruhat decomposition we can also write $g_1^{-1}g_2= b \mathrm{s}^{\epsilon} b'$, where $b,b' \in \mathcal{B}(k)$, and $\epsilon \in \lbrace 0,1 \rbrace$.

Firstly, assume that $\epsilon=0$. Then, it follows from Lemma~\ref{stab visual limit} that  $g_1^{-1}g_2$ belongs to $\mathrm{Stab}_{\mathcal{G}(k)}(\xi_{\infty})$. Hence, we conclude $\xi_1= g_1 \cdot \xi_{\infty}= g_1 (g_1^{-1}g_2) \cdot \xi_{\infty}= g_2 \cdot \xi_{\infty}= \xi_2$. Thus, the result follows in this case.

Now, assume that $\epsilon=1$. It follows from Equation \eqref{eq P subset stab} and Lemma~\ref{stab visual limit} that $P \subseteq g_1\mathcal{B}(k) g_1^{-1} \cap g_2  \mathcal{B}(k) g_2^{-1}$. Since $g_1^{-1}g_2= b \mathrm{s} b'$, with $b,b' \in \mathcal{B}(k)$, we have
$$(g_1b)^{-1} P (g_1b) \subseteq \mathcal{B}(k) \cap \mathrm{s}  \mathcal{B}(k) \mathrm{s}^{-1}.$$
Set $h=bg_1$. Since $\mathcal{B}(k) \cap \mathrm{s}  \mathcal{B}(k) \mathrm{s}^{-1}=\mathcal{T}(k)$ and $P \subset \mathcal{G}(A)$, we deduce that $h^{-1} P h $ is contained in
\begin{equation}\label{eq P conj subset stab}
\mathcal{T}(k) \cap h \mathcal{G}(A) h^{-1}= \left\lbrace \mathrm{diag}(t, \overline{t}t^{-1}, \overline{t}^{-1}): t \in \mathrm{R}_{\ell/k}(\mathbb{G}_{m,\ell}) \right\rbrace \cap h \mathcal{G}(A) h^{-1}.
\end{equation}
Recall that $\mathcal{G}$ admits an embedding $\mathcal{G} \to \mathrm{R}_{\mathcal{D}/\mathcal{C}}(\mathrm{SL}_{3, \mathcal{D}})$. Through this representation, we identify elements in $\mathcal{G}(A)$ with the $3 \times 3$ matrices with coefficients in $B$ (c.f.~\S\ref{subsection radical datum}). Then, the eigenvalues $t, \overline{t}t^{-1}$ and $\overline{t}^{-1}$ of any matrix $\mathrm{diag}(t, \overline{t}t^{-1}, \overline{t}^{-1}) \in h \mathcal{G}(A) h^{-1}$ belong to $B$. Since $\nu_Q(B \smallsetminus \lbrace 0 \rbrace ) \subseteq \mathbb{Z}_{\leq 0}$, we deduce that
$$ 0 \geq \nu_Q(t)= \nu_Q(\overline{t})=-\nu(\overline{t}^{-1}) \geq 0.$$
This implies that $t$ belongs to $ \lbrace x \in B: \nu_Q(x) = 0 \rbrace$, which is an extension $\mathbb{L}$ of $\mathbb{F}$ whose degree is at most two. Thus, we conclude $P \subset h \left\lbrace \mathrm{diag}(t,1,t^{-1}): t \in \mathbb{L}^{*} \right\rbrace h^{-1}$.
Since $\mathbb{L}$ has cardinality $p^{s}$, for some $s \in \mathbb{Z}$, the group $\left\lbrace \mathrm{diag}(t,1,t^{-1}): t \in \mathbb{L}^{*} \right\rbrace \cong \mathbb{L}^{*}$ has a cardinal coprime with $p$. Thus, we obtain a contradiction.
\end{proof}

\section{On the Steinberg module}\label{section Euler-Poincare}
In all this section we assume that $\mathbb{F}=\mathbb{F}_{p^r}$.
We identify the spherical building $\partial_{\infty}(X)$ with $H(\ell, k) \cup \lbrace \infty \rbrace$, according to Corollary~\ref{coro visual limit and H(L,K)}.
Let $\mathbb{Z}[H(\ell, k) \cup \lbrace \infty \rbrace]$ be the free abelian group with basis $H(\ell, k) \cup \lbrace \infty \rbrace$, and let $\epsilon: \mathbb{Z}[H(\ell, k) \cup \lbrace \infty \rbrace] \to \mathbb{Z}$ be the augmentation homomorphism defined from $\epsilon(z)=1$, for all $z \in H(\ell, k) \cup \lbrace \infty \rbrace$.

\begin{definition}\label{def without p' torsion}
We define a group without $p'$-torsion as a group where the order of every element of finite order is a power of $p$.
\end{definition}

Let $G$ be a finite index subgroup of $\Gamma$ without $p'$-torsion.
We define the Steinberg module $\mathrm{St}=\mathrm{St}(G,X)$ induced form the action of $G$ on $X$ as the kernel of $\epsilon$.
Note that $\mathrm{St}$ is a $\mathbb{Z}[G]$-module, since $G \cdot \mathrm{St}=\mathrm{St}$.
Thus, we naturally obtain the following exact sequence of $\mathbb{Z}[G]$-modules:
\begin{equation}\label{Eq ex sequence St}
0 \to \mathrm{St} \to \mathbb{Z}[H(\ell, k) \cup \lbrace \infty \rbrace] \to \mathbb{Z} \to 0.
\end{equation}
In order to prove Theorem~\ref{Main teo steinberg}, we introduce the following definitions.

\begin{definition}\label{defi of unstable}
We say that a vertex (resp. an edge) of $X$ is $G$-stable if its stabilizer is trivial. In any other case, we say that it is $G$-unstable.
We define $X_{\infty}$ as the subgraph of $X$ consisting of its $G$-unstable vertices and edges. This definition makes sense, since, if $e \in \mathrm{E}(X)$ is $G$-unstable, then any vertex $v$ such that $ o(e)=v$ or $t(e)=v$ is also $G$-unstable. Note that $X_{\infty} \neq \emptyset$, according to Theorem~\ref{Teo ABLL stab}(b).
\end{definition}

Since $G$ preserve the type of vertices in $X$, there exists an orientation $\mathrm{E}^{+}(X)$ of $X$ that is invariant under $G$.

\begin{definition}\label{def L_i}
We write
$$S_0= \mathrm{V}(X) \smallsetminus \mathrm{V}(X_{\infty}), \text{ and } S_1= \mathrm{E}^{+}(X) \smallsetminus \mathrm{E}^{+}(X_{\infty}).$$
The group $G$ acts freely on each $S_i$.
For each index $i \in \lbrace 0,1 \rbrace$, we write $L_i$ for the $\mathbb{Z}[G]$-module $L_i= \mathbb{Z}[S_i],$
whose rank over $\mathbb{Z}[G]$ is $l_i= \sharp (G \backslash S_i)$.
\end{definition}

\begin{lemma}
$l_0$ and $l_1$ are finite.
\end{lemma}

\begin{proof}
Assume that $l_0=\infty$. Then, by definition of $l_0$, there exists a sequence $\mathfrak{v}=\lbrace v_i \rbrace_{i=1}^{\infty} \subset \mathrm{V}(X)$ such that $G \cdot v_i \cap G \cdot v_j=\emptyset$, for all $i \neq j$, and $\mathrm{Stab}_G(v_i)=\lbrace e \rbrace$, for all $i \in \mathbb{Z}_{>0}$. Since the stabilizer in $G$ of each vertex $v_i$ is trivial, its stabilizer in $\Gamma$ has at most $N:=[\Gamma:G]$ elements.
The set of such vertices is finite, modulo the action of $\Gamma$, as Corollary~\ref{coro finitud} shows.
Then, there exists a subsequence $\lbrace v_{i_s} \rbrace_{s=1}^{\infty}$ of $\mathfrak{v}$ whose elements belong in the same $\Gamma$-orbit. In particular, for any $s \in \mathbb{Z}_{>0}$ there exists $\gamma_{s} \in \Gamma$ such that $v_{i_s}= \gamma_s \cdot v_{i_1}$. Since $[\Gamma:G]=N < \infty$, there exists a pair $(s_1,s_2)$, with $s_1 \neq s_2$, such that $\gamma_{s_1} \gamma_{s_2}^{-1} \in G$. Thus $G \cdot v_{i_{s_1}} \cap G \cdot v_{i_{s_2}}\neq \emptyset$, which is an absurd. Hence, we conclude that $l_0<\infty$. An analogous argument shows that $l_1$ is also finite.
\end{proof}

Now, we interpret the $\mathbb{Z}[G]$-modules $L_0$ and $L_1$ in terms of homological chains. 
Indeed, for each $i \in \lbrace 0,1\rbrace$ we have that $L_i$ is the group of $i$-dimensional chains in $X$ modulo $X_{\infty}$, i.e. it is the reduced chain complex of $X / X_{\infty}$.
The exact sequence for relative homology gives
\begin{multline*}\label{eq top rel homology}
 \lbrace 0 \rbrace \to  H_{1}(X_{\infty}) \to
 H_{1}(X) \to  H_{1}(X \text{ mod } X_{\infty} )
 \\
 \to H_0(X_{\infty}) 
 \to H_{0}(X) \to  H_{0}(X \text{ mod } X_{\infty} ) \to \lbrace 0 \rbrace.
\end{multline*}
Since $X$ is connected and contractile, we obtain $H_0(X)= \mathbb{Z}$ and $H_1(X)= \lbrace 0 \rbrace$. This implies that 
$  \lbrace 0 \rbrace \to  H_{1}(X \text{ mod } X_{\infty} )
 \to H_0(X_{\infty}) 
 \to \mathbb{Z} \to  H_{0}(X \text{ mod } X_{\infty} ) \to \lbrace 0 \rbrace$ is exact.
 
 \begin{lemma}\label{H_0 relative}
$H_0(X \text{ mod } X_{\infty})=\lbrace 0 \rbrace$ 
 \end{lemma}

 \begin{proof}
Since $H_0(X_{\infty}) \to \mathbb{Z}$ is surjective, the result follows.
 \end{proof}

In particular $\lbrace 0 \rbrace \to H_1(X \text{ mod } X_{\infty}) \to H_0(X_{\infty}) \to \mathbb{Z} \to \lbrace 0 \rbrace$
is exact.
Therefore, the group $H_1(X \text{ mod }X_{\infty})$ may be identified with the kernel of the augmentation homomorphism $H_0(X_{\infty}) \to \mathbb{Z}$. The following lemma proves that this kernel is, in fact, the Steinberg module defined in \eqref{Eq ex sequence St}.

\begin{lemma}\label{lemma H_1}
The $0$-homology group $H_0(X_{\infty})$ is isomorphic to $\mathbb{Z}[H(\ell,k) \cup \lbrace 0 \rbrace]$. In particular, we have $H_1(X \text{ mod }X_{\infty}) \cong \mathrm{St}$.
\end{lemma}

Let $k_P$ be the completion of $k$ at $P$.
Let $v$ be a vertex of $X_{\infty}$. Then, by definition $\mathrm{Stab}_G(v) \neq \lbrace e \rbrace$.
The $\mathcal{G}(k)$-action on $X$ naturally extends to $\mathcal{G}(k_P)$.
Moreover, since $\mathbb{F}$ is finite, the group $\mathrm{Stab}_{\mathcal{G}(k_P)}(v)$ is a compact subgroup of $\mathcal{G}(k_P)$. Then, since $G \subseteq \Gamma$ is discrete, we get that $\mathrm{Stab}_G(v)$ is finite, whence it is a torsion group. Moreover, since $G$ has not $p'$-torsion, we conclude $\mathrm{Stab}_G(v)$ is a non-trivial finite $p$-subgroup of $\mathcal{G}(k)$. Then, it follows from Proposition~\ref{Lemma action of a p-group} that such a group fixes exactly one visual limit $\xi_v$ in $\partial_{\infty}(X)$. 

\begin{definition}\label{def of Q_x} Let $v \in \mathrm{V}(X_{\infty})$. Let us denote by $\mathfrak{r}_v$ be the unique ray in $X$ joining $v$ with the visual limit $\xi_v$ introduced above (c.f.~\S\ref{subsection graphs}). 
\end{definition}

\begin{proof}[Proof of Lemma~\ref{lemma H_1}] The result is a consequence of Corollary~\ref{coro visual limit and H(L,K)} and the following claim. We affirm that the map
\begin{equation}
\Theta: X_{\infty} \to \partial_{\infty}(X), \qquad v \mapsto \xi_v,
\end{equation}
induces a bijection $\Theta'$ between the set of connected component of $X_{\infty}$ onto $\partial_{\infty}(X)$. In order to prove this claim we proceed in three steps.

Firstly, assume that $v$ and $v'$ belongs to the same connected component of $X_{\infty}$. We have to prove that $\xi_v=\xi_{v'}$, i.e. $\Theta'$ is well-defined. Note that, by an inductive argument on the distance between $v$ and $v'$, it suffices to prove the preceding claim when $v$ and $v'$ are adjacent. In this case, let $e$ be an edge in $X_{\infty}$ connecting $v$ with $v'$. Since $e \in \mathrm{E}(X_{\infty})$, we have that $\mathrm{Stab}_G(e)$ is a non-trivial $p$-subgroup of $\Gamma$.
Therefore, it follows from Proposition~\ref{Lemma action of a p-group} that $\mathrm{Stab}_G(e)$ stabilizes one and only one point $\xi_e$ of $\partial_{\infty}(X)$. Since $\mathrm{Stab}_G(e)$ is contained in $\mathrm{Stab}_G(v)$ and $\mathrm{Stab}_G(v') $, we deduce that $\mathrm{Stab}_G(e)$ simultaneously stabilizes $\xi_v$ and $\xi_{v'}$. By the uniqueness of $\xi_e$, we conclude $\xi_v=\xi_e=\xi_{v'}$.

Secondly, assume $v,v' \in X_{\infty}$ satisfy $\xi_v=\xi_{v'}$. Then, we claim that $v$ and $v'$ belong to the same connected component of $X_{\infty}$, i.e. $\Theta'$ is injective. Indeed, if $\xi_v=\xi_{v'}$, then $\mathfrak{r}_v \cap \mathfrak{r}_{v'}$ contains a ray.
This implies that $\mathfrak{r}:=\mathfrak{r}_v \cup \mathfrak{r}_{v'}$ is connected.
Now, since $\mathrm{Stab}_G(v)$ simultaneously stabilizes $v$ and $\xi_v$, then $\mathrm{Stab}_G(v)$ stabilizes $\mathfrak{r}_v$.
In particular $\mathfrak{r}_v \subseteq X_{\infty}$. Analogously, we deduce $\mathfrak{r}_{v'} \subseteq X_{\infty}$. Thus, we conclude $\mathfrak{r}$ is a connected subgraph of $X_{\infty}$ containing $v$ and $v'$. Then, the second claim follows.

Finally, we prove that every element in $\partial_{\infty}(X)$ has the form $\xi_v$, for some $v \in \mathrm{V}( X_{\infty})$, i.e. $\Theta'$ is surjective. Let $\xi$ be a visual limit of $X$.
It follows from Theorem~\ref{Teo ABLL stab}(c) that there exists a ray $\mathfrak{r}_\xi \subset X$ such that  $$\mathrm{Stab}_{\Gamma}(\xi)= \bigcup \left\lbrace \mathrm{Stab}_{\Gamma}(w): w \in \mathrm{V}(\mathfrak{r}_\xi) \right\rbrace.$$
Thus, by intersecting both sides of the previous equation with $G$, we obtain
\begin{equation}\label{eq G-stabilizers} \mathrm{Stab}_{G}(\xi)= \bigcup \left \lbrace \mathrm{Stab}_{G}(w) : w \in \mathrm{V}(\mathfrak{r}_\xi) \right \rbrace.
\end{equation}
Since $[\Gamma:G]$ is finite, we have that $[\mathrm{Stab}_{\Gamma}(\xi):\mathrm{Stab}_{G}(\xi)]$ is finite.
Then, since Theorem~\ref{Teo ABLL stab}(b) shows that $\mathrm{Stab}_{\Gamma}(\xi)$ is infinite, we deduce that $\mathrm{Stab}_{G}(\xi)$ is non-trivial.
Hence, we deduce from Equation \eqref{eq G-stabilizers} that there exists $v \in \mathrm{V}(\mathfrak{r}_\xi)$ whose stabilizer in $G$ is a non-trivial subgroup of $\mathrm{Stab}_G(\xi)$. In particular $v \in \mathrm{V}(X_{\infty})$, whence there exists a unique visual limit $\xi_v$, which is fixed by $\mathrm{Stab}_G(v)$.
Moreover, since $\mathrm{Stab}_G(v) \subseteq \mathrm{Stab}_G(\xi)$, we conclude $\xi_v=\xi$. 
\end{proof}

\begin{proposition}\label{prop char eu}
One has $\chi(G)= l_0-l_1= \mathrm{rk}_{\mathbb{Z}[G]} L_0 -\mathrm{rk}_{\mathbb{Z}[G]} L_1$.
\end{proposition}

\begin{proof}
We decompose the sum introduced in Equation \eqref{eq euler-poincare char} in two parts. One relative to $X_{\infty}$, and other relative to $X$ mod $X_{\infty}$.
We star by study the latter. Indeed, when $v$ (resp. $e$) is a vertex (resp. edge) not contained in $X_{\infty}$, we have $\mathrm{Stab}_G(v)= \lbrace e \rbrace$ (resp. $\mathrm{Stab}_G(v)= \lbrace e \rbrace$). Then, the contribution of $X$ mod $X_{\infty}$ to $\chi(G)$ equals
$$\chi(G)|_{X \text{ mod } X_{\infty} } =  \sharp (G \backslash S_0)-\sharp (G \backslash S_1) = l_0-l_1 .$$
Thus, in order to conclude the result, we just have to show that the contribution of $X_{\infty}$ to $\chi(G)$ is zero. In other words, we have to prove
$$
  \sum_{v \in \mathrm{V}(G \backslash X_{\infty})}  1/\sharp \mathrm{Stab}_G(v)- \sum_{e \in \mathrm{geom(E)}(G \backslash X_{\infty})}  1/ \sharp \mathrm{Stab}_G(e) =0,
$$
or equivalently (c.f.~\S\ref{subsection graphs}),
$$
  \sum_{v \in \mathrm{V}(G \backslash X_{\infty})}  1/\sharp \mathrm{Stab}_G(v)- \sum_{e \in \mathrm{E^{+}}(G \backslash X_{\infty})}  1/ \sharp \mathrm{Stab}_G(e) =0,
$$
where $\mathrm{E^{+}}(G \backslash X_{\infty})=G \backslash \mathrm{E^{+}}(X_{\infty})$, since $G$ respects the orientation on $X$.
Let $v$ be a vertex in $X_{\infty}$. Let $\mathfrak{r}_v$ be the ray introduced in Definition~\ref{def of Q_x}. Let $e=e_v \in \mathrm{E}^{+}(X_{\infty})$ be the unique oriented edge of $\mathfrak{r}_v$ that contains $v$. Since $\mathrm{Stab}_G(v)$ stabilizes $\mathfrak{r}_v$ we get
$$\mathrm{Stab}_G(v) \subseteq \mathrm{Stab}_G(e_v). $$
Moreover, since $G \subseteq \Gamma$ does not change the type of vertices in $X$, we have $\mathrm{Stab}_G(e_v) \subseteq \mathrm{Stab}_G(v)$, whence we deduce $\mathrm{Stab}_G(v)=\mathrm{Stab}_G(e_v)$. Thus, the contribution of $v$ and $e_v$ to $\chi(G)$ cancel. So, in order to conclude the proposition, we just have to prove that the map $\mathrm{V}(X_{\infty}) \to \mathrm{E}^{+}(X_{\infty})$, defined by $v \mapsto e_v$, is a bijection. Set $e \in \mathrm{E}^{+}(X_{\infty})$. Since $\mathrm{Stab}_G(e)$ is a finite non-trivial $p$-group, it fixes a visual limit $\xi_e$ in $\partial_{\infty}(X)$.
Let $\mathfrak{r}_e$ be the smallest ray in $X$ joining $e$ with $\xi_e$ and containing $e$.
Then, the origin $v_e$ of $\mathfrak{r}_e$ is the vertex in $\lbrace o(e), t(e) \rbrace$ ``furthest'' from $\xi_e$. It is not hard to check that the map $ \mathrm{E}^{+}(X_{\infty}) \to \mathrm{V}(X_{\infty})$ defined by $e \mapsto v_e$ is the inverse of the map defined by $v \mapsto e_v$. This concludes the proof.
\end{proof}

\begin{proof}[Proof of Proposition~\ref{Main teo steinberg}] On one hand, by using the definition of relative homology, we have the exact sequence
\begin{multline*}
 \lbrace 0 \rbrace \to H_1( X \text{ mod } X_{\infty}) \to C_1( X \text{ mod } X_{\infty}) \\ \to C_0( X \text{ mod } X_{\infty}) \to H_0( X \text{ mod } X_{\infty}) \to \lbrace 0 \rbrace.
\end{multline*}
On the other hand, it follows from Lemma~\ref{H_0 relative} and \ref{lemma H_1} that $H_0( X \text{ mod } X_{\infty})= \lbrace 0 \rbrace$ and $H_1( X \text{ mod } X_{\infty}) \cong \mathrm{St}$. Since Definition~\ref{def L_i} shows that $L_i=C_i(X \text{ mod } X_{\infty})$, for all $i \in \lbrace 0,1 \rbrace$, we deduce from the preceding exact sequence that
\begin{equation}\label{Ex seq L_i}
\lbrace 0 \rbrace \to \mathrm{St} \to L_1 \to L_0 \to \lbrace 0 \rbrace, 
\end{equation}
is also exact. Moreover, since $L_0$ is free, the exact sequence~\eqref{Ex seq L_i} splits. This implies that $L_1 \cong \mathrm{St} \oplus L_0$. Thus, since Proposition~\ref{prop char eu} shows $ \mathrm{rk}_{\mathbb{Z}[G]} L_1 -\mathrm{rk}_{\mathbb{Z}[G]} L_0 =-\chi(G)$, the result follows.
\end{proof}

\section{Maximal unipotent subgroups}\label{Section maximal unipotent}

Let $\Gamma=\mathcal{G}(A)$ and let $G$ be a finite index subgroup of $\Gamma$. We say that a subgroup of $\mathcal{G}(k)$ is unipotent when it only contains unipotent elements. The preceding definition makes sense since $\mathcal{G}(k)$ admits a representation in the matricial groups $\mathrm{SL}_3(\ell)$.
The main goal of this section is to describe the conjugacy classes of maximal unipotent subgroups in $G$. In order to do this, we do not assume that $\mathbb{F}$ is finite or $G$ is without $p'$-torsion.

\subsection{A proof of Theorem~\ref{main teo unipotent subgroups}}

In order to prove Theorem~\ref{main teo unipotent subgroups} we need some preparation, which we divide into four lemmas.

\begin{lemma}\label{lemma unipotents in G}
Assume that $\mathbb{F}$ is perfect. Let $U$ be an unipotent subgroup of $\mathcal{G}(k)$. Then, there exists $(u,v) \in H(\ell,k)\cup \lbrace \infty \rbrace$ such that $U \subseteq g_{u,v}^{-1} \mathcal{U}_a(k) g_{u,v}$.
\end{lemma}

\begin{proof} Recall that, since $k$ is the function field of a curve, $k$ is a finite extension of $\mathbb{F}(u)$, where $u \in k$ is transcendental over $\mathbb{F}$.
Let us write $k=\mathbb{F}(u, \theta_1, \cdots,\theta_r)$, where $\theta_i$ is algebraic over $\mathbb{F}(u)$.
Since $\mathbb{F}$ is perfect, we have $k^p=\mathbb{F}(u^p, \theta_1^p, \cdots,\theta_r^p)$.
Let $f_i(T)$ (resp. $g_i(T)$) be the irreducible polynomial of $\theta_i$ (resp. $\theta_i^p$) in $L_i:=\mathbb{F}(u, \theta_1, \cdots,\theta_{i-1})$ (resp. in $F_i:=\mathbb{F}(u^p, \theta_1^p, \cdots,\theta_{i-1}^p)$). 
On the one hand, since $0=g_i(\theta_i^p)=(\tilde{g}_i(\theta_i))^p$, for some $\tilde{g}_i \in L_i[T]$ with the same degree of $g_i$, we have $\deg(g_i) \geq \deg(f_i)$.
On the other hand, since $0=f_i(\theta_i)^p=\tilde{f}_i(\theta_i^p)$, for $\tilde{f}_i \in F_i[T]$ with the same degree of $f_i$, we have $\deg(g_i) \leq \deg(f_i)$.
Thus 
$$[\mathbb{F}(u, \theta_1, \cdots,\theta_r): \mathbb{F}(u)]=\prod_{i=1}^r \deg(f_i)=\prod_{i=1}^r \deg(g_i)=[\mathbb{F}(u^p, \theta_1^p, \cdots,\theta_r^p): \mathbb{F}(u^p)],$$
whence $[k:k^p]=[\mathbb{F}(u):\mathbb{F}(u^p)]=p$.

Let $U$ be a unipotent subgroup of $\mathcal{G}(k)$. Since $\mathcal{G}$ is a semi-simple simply connected $k$-group and $[k:k^p] \leq p$, it follows from \cite[Theorem 2]{Gille} that $U$ is $k$-embeddable in the unipotent radical $\mathcal{R}_u(\mathcal{P})$ of a $k$-parabolic subgroup $\mathcal{P}$ of $\mathcal{G}$. In particular, $U$ is contained in $\mathcal{R}_u(\mathcal{P})(k)$. Moreover, since $\mathcal{R}_u(\mathcal{P})$ is contained in the unipotent radical of some Borel subgroup of $\mathcal{G}$, and two Borel subgroup belong in the same $\mathcal{G}(k)$-conjugacy class, we conclude $U \subseteq \tau^{-1} \mathcal{U}_a(k) \tau$, for some $\tau \in \mathcal{G}(k)$.

By the Bruhat decomposition, we write $\tau=b \mathrm{s}^{\epsilon}\mathrm{u}_a(u,v)$, where $b \in \mathcal{B}(k)$ and $\epsilon \in \lbrace 0,1 \rbrace$.
Since $\mathcal{B}(k)$ normalizes $\mathcal{U}_a(k)$, we have $\tau^{-1} \mathcal{U}_a(k) \tau=\mathcal{U}_a(k)$, if $\epsilon=0$, and $\tau^{-1} \mathcal{U}_a(k)\tau=(\mathrm{s}\mathrm{u}_a(u,v))^{-1} \mathcal{U}_a(k) (\mathrm{s}\mathrm{u}_a(u,v))$, if $\epsilon=1$.
Thus, if we write $g_{u,v}=\mathrm{id}$, if $\epsilon=0$, and $g_{u,v}=\mathrm{s} \mathrm{u}_a(u,v)$, if $\epsilon=1$, then $\tau^{-1} \mathcal{U}_a(k) \tau = g_{u, v}^{-1} \mathcal{U}_a(k) g_{u,v} $.
Hence, we conclude that $U$ is contained in $g_{u,v}^{-1} \mathcal{U}_a(k) g_{u,v}$, for some $(u,v) \in H(\ell,k) \cup \lbrace \infty \rbrace$.
\end{proof}

\begin{lemma}\label{lemma fin index unip quot}
For each $(u,v) \in H(\ell,k) \cup \lbrace \infty \rbrace$, and each $G$ of finite index in $\Gamma$, let
$$\tilde{G}_{u,v}:=(g_{u,v}^{-1} \mathcal{U}_a(k) g_{u,v} \cap G)/(g_{u,v}^{-1} \mathcal{U}_{2a}(k) g_{u,v} \cap G).$$
Then, $ \tilde{G}_{u,v}$ is a finite index subgroup of $\tilde{\Gamma}_{u,v}$. 
In particular, $(g_{u,v}^{-1} \mathcal{U}_a(k) g_{u,v} \cap G) \smallsetminus (g_{u,v}^{-1} \mathcal{U}_{2a}(k) g_{u,v} \cap G)$ is non empty.

\end{lemma}

\begin{proof}
Note that $Q_{u,v}:=(g_{u,v}^{-1} \mathcal{U}_a(k) g_{u,v} \cap \Gamma)/(g_{u,v}^{-1} \mathcal{U}_{a}(k) g_{u,v} \cap G)$ injects into the set $\Gamma/G$. Then $Q_{u,v}$ is finite. 
Now, let $f: g_{u,v}^{-1} \mathcal{U}_a(k) g_{u,v} \cap \Gamma \to \tilde{\Gamma}_{u,v}$ be the canonical projection. Since $f(g_{u,v}^{-1} \mathcal{U}_a(k) g_{u,v} \cap G)=\tilde{G}_{u,v}$, there exists a surjective function from $Q_{u,v}$ onto $\tilde{\Gamma}_{u,v}/\tilde{G}_{u,v}$. 
Thus, the first statement follows.

Note that, if $(g_{u,v}^{-1} \mathcal{U}_a(k) g_{u,v} \cap G) \smallsetminus (g_{u,v}^{-1} \mathcal{U}_{2a}(k) g_{u,v} \cap G)=\emptyset$, then $\tilde{G}_{u,v}=\lbrace \mathrm{id} \rbrace$.
This implies that $\tilde{\Gamma}_{u,v}$ is finite.
But, $\tilde{\Gamma}_{u,v}$ is parametrized by $H(u,v)/H(u,v)^{0}$, which is isomorphic to an $A$-submodule of $\ell$ containing a non-zero ideal $I \subseteq B$ according to Theorem~\ref{Teo ABLL stab}(g). In particular, $\tilde{\Gamma}_{u,v}$ is infinite.
\end{proof}

\begin{lemma}\label{lemma normalizacion de borel}
Let $g \in \mathcal{G}(k)$ such that $g \mathrm{u}_a(x,y) g^{-1} \in \mathcal{U}_a$, for some $\mathrm{u}_a(x,y) \in \mathcal{U}_a(k) \smallsetminus \mathcal{U}_{2a}(k)$. Then $g \in \mathcal{B}(k)$.
\end{lemma}

\begin{proof}
Let $g \in \mathcal{G}(k)$ such that $g \mathrm{u}_a(x,y) g^{-1} \in \mathcal{U}_a(k)$, for some $\mathrm{u}_a(x,y) \in \mathcal{U}_a(k) \smallsetminus \mathcal{U}_{2a}(k)$. Then $(x,y) \in H(\ell,k)$ satisfies $x \neq 0$.
 Since $\mathcal{G}(k)$ embeds into $\mathrm{SL}_3(\ell)$, we can write $g=(a_{i,j})_{i,j=1}^3 \in \mathrm{SL}_3(\ell)$. Moreover, since $g \mathrm{u}_a(x,y) g^{-1} \in \mathcal{U}_a(k)$ , there exists $(z,w) \in H(\ell,k)$ such that $g\mathrm{u}_a(x,y)g^{-1}= \mathrm{u}_a(z,w)$. In other words, we have
$$
\textnormal{\footnotesize$\begin{pmatrix}
        a_{11}    &  a_{12}  &   a_{13}  \\
        a_{21}   &   a_{22}   & a_{23}    \\
        a_{31}   & a_{32} &   a_{33} \\
    \end{pmatrix}$\normalsize} \textnormal{\footnotesize$\begin{pmatrix}
        1    &  -\overline{x}  &   y \\
        0   &   1   & x    \\
        0   & 0 &  1 \\
    \end{pmatrix}$\normalsize} =
    \textnormal{\footnotesize$\begin{pmatrix}
        1    &  -\overline{z}  &   w \\
        0   &   1   &  z   \\
        0   & 0 &  1 \\
    \end{pmatrix}$\normalsize} \textnormal{\footnotesize$\begin{pmatrix}
        a_{11}    &  a_{12}  &   a_{13}  \\
        a_{21}   &   a_{22}   & a_{23}    \\
        a_{31}   & a_{32} &   a_{33} \\
    \end{pmatrix}$\normalsize}.
$$
Then, the matrix
$$
    \textnormal{\footnotesize$\begin{pmatrix}
        a_{11}    &  a_{12}-\overline{x}a_{11}  &   a_{13}+x a_{12}+y a_{11}  \\
        a_{21}   &   a_{22}-\overline{x}a_{21}   & a_{23}+ x a_{22}+ y a_{21}   \\
        a_{31}   & a_{32}-\overline{x}a_{31} &   a_{33} + x a_{32}+ y a_{31} \\
    \end{pmatrix}$\normalsize} $$
equals
$$    \textnormal{\footnotesize$\begin{pmatrix}
        a_{11}- \overline{z} a_{21}+ w a_{31}    &   a_{12}- \overline{z} a_{22}+ w a_{32} &    a_{13}- \overline{z} a_{23}+ w a_{33}  \\
        a_{21}+ z a_{31}   &   a_{22}+ z a_{32}   & a_{23}+ z a_{33}    \\
        a_{31}   & a_{32} &   a_{33}\\
    \end{pmatrix}$\normalsize}.
$$

Thus, if we analyze the coordinates in the last row of the previous matrices, since $x \neq 0$, we deduce $a_{31}=0$, and then $a_{32}=0$. Note that, if $z=0$ (and hence $w\neq 0$), then, by comparing the second columns of the preceding matrices, we obtain that $a_{11}=a_{21}=0$, which contradicts the fact that $g$ is invertible. Hence, we may assume that $z \neq 0$. We deduce then, by comparing the first columns of the previous matrices, that $a_{21}=0$. Thus $g$ is a upper triangular matrix, whence we conclude that $g \in \mathcal{B}(k)$.
\end{proof}

\begin{lemma}\label{lemma cont unip}
Let $(u_1, v_1), (u_2,v_2)$ be two pairs in  $H(\ell,k) \cup \lbrace \infty \rbrace$.
Then $g_{u_1, v_1}^{-1} \mathcal{U}_a(k) g_{u_1,v_1} \cap G$ is contained in $g_{u_2, v_2}^{-1} \mathcal{U}_a(k) g_{u_2,v_2} \cap G$ precisely when $g_{u_2,v_2} g_{u_1,v_1}^{-1}$ belongs to $\mathcal{B}(k)$.
Moreover, in both cases $g_{u_1, v_1}^{-1} \mathcal{U}_a(k) g_{u_1,v_1} \cap G$ equals $g_{u_2, v_2}^{-1} \mathcal{U}_a(k) g_{u_2,v_2} \cap G$.
\end{lemma}

\begin{proof}
Since $\mathcal{B}(k)$ normalizes $\mathcal{U}_a(k)$, if $g_{u_2,v_2} g_{u_1,v_1}^{-1}$ belongs to $\mathcal{B}(k)$, then the group $g_{u_1, v_1}^{-1} \mathcal{U}_a(k) g_{u_1,v_1} \cap G$ equals $g_{u_2, v_2}^{-1} \mathcal{U}_a(k) g_{u_2,v_2} \cap G$. Now, we prove the converse. Indeed, it follows from Lemma~\ref{lemma fin index unip quot} that there exists $\rho:=g_{u_1, v_1}^{-1} \mathrm{u}_a(x,y) g_{u_1, v_1}$ which belongs to $(g_{u_1, v_1}^{-1} \mathcal{U}_a(k) g_{u_1,v_1} \cap G) \smallsetminus (g_{u_1, v_1}^{-1} \mathcal{U}_{2a}(k) g_{u_1,v_1} \cap G)$.
Then, if $g_{u_1, v_1}^{-1} \mathcal{U}_a(k) g_{u_1,v_1} \cap G$ is contained in $g_{u_2, v_2}^{-1} \mathcal{U}_a(k) g_{u_2,v_2} \cap G$, the element $\rho$ belongs to $g_{u_2, v_2}^{-1} \mathcal{U}_a(k) g_{u_2,v_2} \cap G$. This implies that the element $g:=g_{u_2, v_2}g_{u_1, v_1}^{-1}$ satisfies $g \mathrm{u}_a(x,y)g^{-1} \in \mathcal{U}_a(k)$, where $\mathrm{u}_a(x,y) \in \mathcal{U}_a(k) \smallsetminus \mathcal{U}_{2a}(k)$. Thus $g \in \mathcal{B}(k)$ according to Lemma~\ref{lemma normalizacion de borel}.
\end{proof}

\begin{proof}[Proof of Theorem~\ref{main teo unipotent subgroups}]
To start, we check that each $\mathcal{U}(G_{\sigma})$ is maximal among unipotent subgroups.
In order to do this, assume that $\mathcal{U}(G_{\sigma})$ is contained in some unipotent subgroup $U$ of $G$.
It follows from Lemma~\ref{lemma unipotents in G} that $U$ is contained in $g_{u,v}^{-1} \mathcal{U}_a(k) g_{u,v} \cap G$, for some $(u,v) \in H(\ell,k) \cup \lbrace \infty \rbrace$.
Moreover, it follows from the definition of $\mathcal{U}(G_{\sigma})$ that $\mathcal{U}(G_{\sigma})=g_{u_1, v_1}^{-1} \mathcal{U}_a(k) g_{u_1,v_1} \cap G$, for some $(u_1,v_1) \in H(\ell,k) \cup \lbrace \infty \rbrace$. Then, we have
$$ g_{u_1, v_1}^{-1} \mathcal{U}_a(k) g_{u_1,v_1} \cap G \subseteq U \subseteq g_{u, v}^{-1} \mathcal{U}_a(k) g_{u,v} \cap G.$$
Hence $g_{u, v}^{-1} \mathcal{U}_a(k) g_{u,v}= g_{u_1, v_1}^{-1} \mathcal{U}_a(k) g_{u_1,v_1}$, according to Lemma~\ref{lemma cont unip}. Thus, we conclude
$$ g_{u_1, v_1}^{-1} \mathcal{U}_a(k) g_{u_1,v_1} \cap G = U = g_{u, v}^{-1} \mathcal{U}_a(k) g_{u,v} \cap G ,$$ whence the maximality of $\mathcal{U}(G_{\sigma})$ follows.

We are able to prove statement~(1). 
Indeed, since $\mathcal{U}(G_{\sigma}) \subset G_{\sigma}$ is unipotent, it is contained in the group $\mathcal{U}_{\sigma}$ containing all the unipotent element in $G_{\sigma}$.
But, since $\mathcal{U}(G_{\sigma})$ is maximal, we have $\mathcal{U}(G_{\sigma})=\mathcal{U}_{\sigma}$.
Let $\xi=\xi_{\sigma}\in \partial_{\infty}(X)$ be a representative of $\sigma \in \Sigma_0$. Then, the stabilizer $G_{\sigma}= G_{\xi_{\sigma}}$ of $\xi_{\sigma}$ in $G$ is the intersection of $G$ with the stabilizer $\Gamma_{\xi}=\Gamma_{\xi_{\sigma}}$ of $\xi_{\sigma}$ in $\Gamma$.
Thus, we have an injective morphism $G_{\sigma}/\mathcal{U}(G_{\sigma}) \to \Gamma_{\sigma}/\mathcal{U}(\Gamma_{\sigma})$. 
Since, by definition $\mathcal{U}(\Gamma_{\sigma})=U_{\xi_{\sigma}}$, we deduce from Theorem~\ref{Teo ABLL stab}(e) that $G_{\sigma}/\mathcal{U}(G_{\sigma})$ is isomorphic to a subgroup of $\mathbb{F}^{*}$.

Now, we prove statement~(2).
Firstly, let $U$ be a maximal unipotent subgroup in $G$.  It follows from Lemma~\ref{lemma unipotents in G} that $U$ is contained in $g_{u,v}^{-1} \mathcal{U}_a(k) g_{u,v} \cap G$, for some $(u,v) \in H(\ell,k) \cup \lbrace \infty \rbrace$. Since $U$ is maximal, we get $U=g_{u,v}^{-1} \mathcal{U}_a(k) g_{u,v} \cap G$.
Let $\xi=g_{u,v}^{-1} \cdot \xi_{\infty}$ be a visual limit in $\partial_{\infty}(X)$. Then, by definition of the set $\lbrace \xi_{\sigma} \rbrace_{\sigma \in \Sigma_0}$ in Theorem~\ref{main teo unipotent subgroups}, there exists $g \in G$ and $\sigma \in \Sigma_0$ such that $\xi=g \cdot \xi_{\sigma}$. Therefore $U=g \mathcal{U}(G_{\sigma})g^{-1}$. Thus, $U$ and $\mathcal{U}(G_{\sigma})$ belong to the same $G$-conjugacy class.

Secondly, we prove that, given two group $\mathcal{U}(G_{\sigma_1})$ and $\mathcal{U}(G_{\sigma_2})$, with $\sigma_1 \neq \sigma_2$, they are not $G$-conjugated.
Assume that there exists $g \in G$ such that $ g \mathcal{U}(G_{\sigma_1}) g^{-1}= \mathcal{U}(G_{\sigma_2})$.
For each $i \in \lbrace 1,2 \rbrace$, let $\xi_{\sigma_i} \in \partial_{\infty}(X)$ be the visual limit corresponding to $\sigma_i$, $(u_i,v_i)=f^{-1}(\xi_{\sigma_i}) \in H(\ell,k) \cup \lbrace \infty \rbrace$, and $g_{u_i,v_i} \in \mathcal{G}(k)$ such that $\xi_{\sigma_i}=g_{u_i,v_i}^{-1} \cdot \xi_{\infty}$ (c.f.~Corollary~\ref{coro visual limit and H(L,K)}).  
Let us write $\xi=g \cdot \xi_{\sigma_1} \in \partial_{\infty}(X)$. Let $(u,v)\in H(\ell,k) \cup \lbrace \infty \rbrace$ such that $\xi=g_{u,v}^{-1} \cdot \xi_{\infty}$, and write $\mathcal{U}(\xi)=g_{u,v}^{-1} \mathcal{U}_a(k) g_{u,v} \cap G$. Since $g g_{u_1,v_1} \cdot \xi_{\infty}= g \cdot \xi_{\sigma_1}= \xi = g_{u,v}^{-1} \cdot \xi_{\infty}$, it follows from Lemma~\ref{stab visual limit} that $g_{u,v} g g_{u_1,v_1}^{-1} \in \mathcal{B}(k)$. Moreover, since $\mathcal{B}(k)$ stabilizes $\mathcal{U}_a(k)$, we have
$$\mathcal{U}(\xi) = g \mathcal{U}(G_{\sigma_1}) g^{-1}= \mathcal{U}(G_{\sigma_2}).$$
In other words, we have
$$ g_{u,v}^{-1} \mathcal{U}_a(k) g_{u,v} \cap G= g_{u_2,v_2}^{-1} \mathcal{U}_a(k) g_{u_2,v_2} \cap G. $$
Thus, it follows from Lemma~\ref{lemma cont unip} that $g_{u_2,v_2} g^{-1}_{u,v} \in \mathcal{B}(k)$. Then, we deduce from Lemma~\ref{stab visual limit} that $\xi=g^{-1}_{u,v} \cdot \xi_{\infty}=g^{-1}_{u_2,v_2} \cdot \xi_{\infty}=\xi_{\sigma_2}$. In other words, we conclude $\xi_{\sigma_2}= g \cdot \xi_{\sigma_1}$, which contradicts the definition of the set $\lbrace G_{\sigma} \rbrace_{\sigma \in \Sigma_0}$.
Thus, statement~(2) follows.

Finally, in order to prove statement~(3), we assume that $\mathbb{F}=\mathbb{F}_{p^r}$ and $G$ does not have $p'$-torsion.
Let $g \in G_{\sigma}$. 
Since $\mathbb{F}^{*}$ has order $p^r-1$, we have $g^{p^r-1}\in \mathcal{U}(G_{\sigma})$.
In other words, $g^{p^r-1}= g_{u,v}^{-1} \mathrm{u}_a(x,y) g_{u,v}$, for some pair $(x,y) \in H(\ell,k)$.
Thus, the element $\tau:= g_{u,v} g^{p^r-1} g_{u,v}^{-1}$ satisfies $\tau^p= \mathrm{u}_a(px,z)= \mathrm{u}_a(0,z)$, for some $z \in H^0(\ell,k)$.
This implies that $\tau^{p^2}= \mathrm{u}_a(0,z)^p=\mathrm{u}_a(0,pz)=\mathrm{id}$.
In particular, the order $|g|$ of $g$ divides $p^2\left(p^r-1\right)$.
Since torsion elements in $G$ have order a power of $p$, we deduce that $|g|$ divides $p^2$.
Let $a,b \in \mathbb{Z}$ such that $ap^2+b(p^r-1)=1$.
Then $g=(g^{p^2})^a (g^{p^r-1})^b=(g^{p^r-1})^b$ belongs to $\mathcal{U}(G_{\sigma})$, whence the result follows.
\end{proof}

\subsection{On some consequences}\label{subsection consequences of main teo max unip sub}

The following results are some interesting consequences of Theorem \ref{main teo unipotent subgroups}.

\begin{corollary}\label{coro number of max unip sub for GA}
Assume that $\mathbb{F}$ is perfect. Then, there exists a bijection between the conjugacy classes of maximal unipotent subgroups in $\Gamma$ and $\mathrm{Pic}(B)$.
\end{corollary}

\begin{proof}
It follows from Theorem~\ref{main teo unipotent subgroups} that there exists a bijection between the conjugacy classes of maximal unipotent subgroups in $\Gamma$ and the $\Gamma$-orbits in $\partial_{\infty}(X)$. Then, the result directly follows from Theorem~\ref{Teo ABLL quotient}.
\end{proof}

\begin{corollary}\label{coro max unip sub A principal}
Assume that $\mathbb{F}$ is perfect. Then, the function ring $B$ is a principal ideal domain if and only if each unipotent subgroup of $\Gamma$ is contained in a $\Gamma$-conjugated of $\mathcal{U}_a(k)\cap \Gamma$.
\end{corollary}

\begin{proof}
Since $B$ is a Dedekind domain, $\mathrm{Pic}(B)$ is trivial exactly when $B$ is a principal ideal domain. Then, Theorem~\ref{Teo ABLL quotient} shows that $\mathrm{Pic}(B)$ is trivial if and only if $\Gamma$ acts transitively on $\partial_{\infty}(X)$. Thus, it follows from Theorem~\ref{main teo unipotent subgroups} that $\mathrm{Pic}(B)$ is trivial precisely when $\Gamma$ has a unique class of maximal unipotent subgroups. Moreover, if $B$ is principal, then, we can chose $\xi_{\infty}$ as a representative of the unique $\Gamma$-orbit in $\partial_{\infty}(X)$. Then, Theorem~\ref{main teo unipotent subgroups} implies that $\mathcal{U}_a(k)\cap \Gamma$ represent the unique conjugacy class of maximal unipotent subgroups in $\Gamma$. Thus, since any unipotent subgroup of $\Gamma$ is contained in a maximal unipotent subgroup according to \ref{lemma unipotents in G}, the result follows.
\end{proof}


\begin{corollary}\label{coro max unip sub in sub of GA}
Assume that $\mathbb{F}$ is finite. Then $G$ has at most $[\Gamma:G] \cdot \sharp \mathrm{Pic}(B) < \infty$ conjugacy classes of maximal unipotent subgroups. 
\end{corollary}

\begin{proof}
It is not hard to see that the number of $G$-orbits in $\partial_{\infty}(X)$ is lees or equal than the number of $\Gamma$-orbits in $\partial_{\infty}(X)$ multiplied by the index $[\Gamma:G]$. Thus, $G$ has at most $[\Gamma:G] \cdot \sharp \mathrm{Pic}(B)$ conjugacy classes of maximal unipotent subgroups, according to Theorem~\ref{Teo ABLL quotient}. Since $\mathbb{F}$ is finite, $\mathrm{Pic}(B)$ is finite according to Weyl theorem (c.f.~\cite[Ch. II, \S 2.2]{S}).
\end{proof}

The next result follows from the description of $G$ as the amalgamated free product defined from $G \backslash X$ by using the Bass-Serre theory (c.f.~\cite[Chapter I, \S 5]{S} and \cite[Chapter II, \S 2.5]{S}).

\begin{corollary}
Assume that $\mathbb{F}=\mathbb{F}_{p^r}$.
Let $G$ be a finite index subgroup of $\Gamma$ without $p'$-torsion. Then, $G$ is the free product of a finite generated group $K$ with a representative system of its maximal unipotent subgroups $\lbrace G_{\sigma} \rbrace_{\sigma \in \Sigma_0}$, amalgamated along certain groups $K_{\sigma}$, according to certain injections $K_{\sigma} \hookrightarrow G_\sigma$ and $K_{\sigma} \hookrightarrow K$.
\end{corollary}

\section{Relative Homology}\label{Rel. Homology}
In all this section we assume that $\mathbb{F}=\mathbb{F}_{p^r}$.
We start by recalling the definition of the relative homology of a group $H$ respect to a non-empty family of subgroups $\lbrace H_{\sigma} \rbrace_{\sigma \in \Sigma}$. Indeed, let $R$ be the kernel of the augmentation $\mathbb{Z}$-morphism $\coprod_{\sigma \in \Sigma} \mathbb{Z}[H/H_{\sigma}] \to \mathbb{Z}$, defined by $\varepsilon(\delta)=1$, for all $\delta \in H/H_{\sigma}$ and all $\sigma \in \Sigma$. Then, we get the exact sequence:
\begin{equation}\label{eq relative hom}
 0 \to R \to \coprod_{\sigma \in \Sigma} \mathbb{Z}[H/H_{\sigma}] \to \mathbb{Z} \to 0.
\end{equation} 
Let $M$ be a $H$-module, and, for each $i \in \mathbb{Z}_{\geq 1}$, let us write
$$H_i(H \text{ mod }H_{\sigma},M):= \mathrm{Tor}_{i-1}^{\mathbb{Z}[H]}(R,M)= H_{i-1}(H, R \otimes M). $$
This groups are called the $M$-valued homology groups of $H$ relative to the family of subgroups $\lbrace H_{\sigma} \rbrace_{\sigma \in \Sigma}$.

\begin{theorem}\label{teo relative homology}
Let $G$ be a finite index subgroup of $\Gamma$ without $p'$-torsion.
Then, the relative homology groups of $G$ modulo the family $\lbrace G_{\sigma} \rbrace_{\sigma \in \Sigma_0}$ of unipotent subgroups defined in Theorem~\ref{main teo unipotent subgroups} satisfy:
\begin{itemize}
    \item[(a)] $H_{i}(G \text{ mod } G_{\sigma},M)= \lbrace 0 \rbrace$, for all $i \neq 1$, and
    \item[(b)] If $M$ es finitely generated over $\mathbb{Z}$, then $H_{1}(G \text{ mod } G_{\sigma},M) \cong M^{-\chi(G)}$.
\end{itemize}
\end{theorem}

\begin{proof}
It follows from the orbit-stabilizer relation that $$\mathbb{Z}[H(\ell,k) \cup \lbrace \infty \rbrace] = \coprod_{\sigma \in \Sigma_0} \mathbb{Z}[G/G_{\sigma}].$$
Then, the module $R$ defined in Equation~\eqref{eq relative hom} is none other than the Steinberg module defined in Equation~\eqref{Eq ex sequence St}. Thus, for each $G$-module $M$, we have
\begin{equation}
    H_i(G \text{ mod } G_{\sigma}, M)= \mathrm{Tor}_{i-1}^{\mathbb{Z}[G]}(\mathrm{St},M).
\end{equation}
Now, recall that, for all projective module $P$, we have $\mathrm{Tor}_{i-1}(P,M)=0$, for all $i \neq 1$. Then, assertion (a) follows from from Theorem~\ref{Main teo steinberg}. Now, we just to have to prove (b). To do so, recall that 
$$H_1(G \text{ mod } G_{\sigma},M)= \mathrm{Tor}_0^{\mathbb{Z}[G]}(\mathrm{St}, M)= \mathrm{St}  \otimes_{\mathbb{Z}[G]} M.$$
Then, it follows from Theorem~\ref{Main teo steinberg} that
$$H_1(G \text{ mod } G_{\sigma},M) \oplus M^{l_0} \cong M^{l_1}, \text{ with } \chi(G)=l_0-l_1. $$
Since $M$ is finitely generated, it follows from an invariant factor argument that $H_1(G \text{ mod } G_{\sigma},M)$ is isomorphic to $M^{l_1-l_0}= M^{-\chi(G)}$.
\end{proof}

\begin{corollary}\label{coro homology of G}
For each $i \in \mathbb{Z}_{>1} $ we have $H_{i}(G,M) \cong \coprod_{\sigma \in \Sigma_0} H_i(G_{\sigma},M)$. For $i=1$, the following sequence is exact:
$$
0 \to \coprod_{\sigma \in \Sigma_0} H_1(G_{\sigma},M) \to H_1 (G,M) \to H_1(G \text{ mod } G_{\sigma},M),$$
where $H_1(G \text{ mod } G_{\sigma},M)$ can be replaced by $M^{-\chi(G)}$ when $M$ is $\mathbb{Z}$-finitely generated.
\end{corollary}

\begin{proof}
The $\mathrm{Tor}$-exact sequence, together with Shapiro's Lemma, give the exact sequence:
$$\cdots \to H_{i+1}( G \text{ mod } G_{\sigma}, M) \to 
\coprod_{\sigma \in \Sigma_0} H_i(G_{\sigma},M) \to
H_i (G,M) \to \cdots .$$
Then, result follows from Theorem~\ref{teo relative homology}.
\end{proof}

For any group $F$, let us denote by $F^{\mathrm{ab}}$ its largest abelian quotient.
As in Lemma~\ref{lemma fix points in V}, given $(u,v) \in H(\ell,k) \cup \lbrace \infty \rbrace$, we write $\tilde{G}_{u,v}=(g_{u,v}^{-1}\mathcal{U}_{a}(k) g_{u,v} \cap G) / (g_{u,v}^{-1}\mathcal{U}_{2a}(k) g_{u,v} \cap G)$.
As in Theorem~\ref{main teo unipotent subgroups}, let $\lbrace \xi_{\sigma} \rbrace_{\sigma \in \Sigma_0}$ be a set that represents the $G$-orbits in $\partial_{\infty}(X)$. 
For each $\sigma \in \Sigma_0$, let $(u,v)=f^{-1}(\xi_{\sigma}) \in H(\ell,k) \cup \lbrace \infty \rbrace$. By abuse of notation, we denote by $\Sigma_0$ the counting set for the pairs $(u,v)$ defined above.

\begin{proposition}\label{prop abelianization of G}
Let $\mathrm{Tor}\left( G^{\mathrm{ab}} \right)$ be the torsion subgroup of
$G^{\mathrm{ab}}$. Then
\begin{itemize}
    \item[(a)] $\mathrm{Tor}\left( G^{\mathrm{ab}} \right)$ is isomorphic to $\coprod_{(u,v) \in \Sigma_0} \tilde{G}_{u,v}$,
    \item[(b)] the $\mathbb{Z}$-rank of $G^{\mathrm{ab}}/\mathrm{Tor}\left( G^{\mathrm{ab}} \right)$ is less or equal than $-\chi(G)$, and
    \item[(c)] there exists a family $\lbrace I_{u,v}: (u,v) \in \Sigma_0 \rbrace$ of non-trivial $B$-fractional ideals and a natural isomorphism between $\mathrm{Tor}\left(G^{\mathrm{ab}}\right)$ and a finite index subgroup of $\coprod_{(u,v) \in \Sigma_0} I_{u,v}$.
\end{itemize}
In particular, $G^{\mathrm{ab}}$ is not finitely generated.
\end{proposition}

\begin{proof}
It follows from Corollary~\ref{coro homology of G} that $0 \to \coprod_{\sigma \in \Sigma_0} G_{\sigma}^{\mathrm{ab}} \to G^{\mathrm{ab}} \to \mathbb{Z}^{-\chi(G)}$ is exact. 
Since each element in $G_{\sigma}$ is killed by $p^2$ and $\mathbb{Z}^{-\chi(G)}$ is torsion free, we have
$\mathrm{Tor}\left( G^{\mathrm{ab}} \right) \cong \coprod_{\sigma \in \Sigma_0} G_{\sigma}^{\mathrm{ab}}$ and $\mathrm{rk}_{\mathbb{Z}} \left( G^{\mathrm{ab}}/\mathrm{Tor}\left( G^{\mathrm{ab}} \right) \right) \leq -\chi(G)$.
Thus, statement~(b) follows.

Recall that, for any $\sigma$, there exists $(u,v)\in H(\ell,k) \cup \lbrace \infty \rbrace$ such that $G_{\sigma}= g_{u,v}^{-1}\mathcal{U}_a(k) g_{u,v} \cap G$.
Since, for any $(u,v), (x,y) \in H(\ell,k)$ we have $[\mathrm{u}_a(u,v), \mathrm{u}_a(x,y)]= \mathrm{u}_a(0,u\overline{x}-\overline{u}x)$, the commutator $[G_{\sigma}, G_{\sigma}]$ is contained in $g_{u,v}^{-1}\mathcal{U}_{2a}(k) g_{u,v} \cap G$.
Moreover, since $(g_{u,v}^{-1}\mathcal{U}_{a}(k) g_{u,v} \cap G) / (g_{u,v}^{-1}\mathcal{U}_{2a}(k) g_{u,v} \cap G)$ is abelian, because it is isomorphic to a subgroup of $\ell$, we conclude that $[G_{\sigma}, G_{\sigma}]=g_{u,v}^{-1}\mathcal{U}_{2a}(k) g_{u,v} \cap G$.
In other words, $G_{\sigma}^{\mathrm{ab}}=(g_{u,v}^{-1}\mathcal{U}_{a}(k) g_{u,v} \cap G) / (g_{u,v}^{-1}\mathcal{U}_{2a}(k) g_{u,v} \cap G)= \tilde{G}_{u,v}$.
Thus, statement~(a) follows.

Now, it follows from Lemma~\ref{lemma fin index unip quot} that $\tilde{G}_{u,v}$ is a finite index subgroup of $\tilde{\Gamma}_{u,v}$.
Note that $\tilde{\Gamma}_{u,v}$ is isomorphic to $H(u,v)/H(u,v)^{0}$ as a group.
Since $H(u,v)/H(u,v)^{0}$ is a non-trivial finitely generated $A$-module of $\ell$ according to Theorem~\ref{Teo ABLL stab}(g), we can write $H(u,v)/H(u,v)^{0}=A x_1+ \cdots A x_r$, for some $x_1, \cdots, x_r \in \ell$.
Let us write $x_i=b_i/b_i'$ with $b_i, b_i'\in B$ and $b=\prod_{i=1}^r b_i$.
Then $H(u,v)/H(u,v)^{0} \subseteq b^{-1}B$.
Set $I_{u,v}=b^{-1}B$.
Recall that $H(u,v)/H(u,v)^{0}$ contains a $B$-ideal, according to Theorem~\ref{Teo ABLL stab}(g). Hence $H(u,v)/H(u,v)^{0}$ has finite index in $I_{u,v}$.
Thus, $\tilde{G}_{u,v}$ is isomorphic to a finite index subgroup of $I_{u,v}$. Then, statement~(c) holds.

Finally, assume that $G^{\mathrm{ab}}$ is finitely generated. 
Since $\sharp(\Sigma_0) \leq [\Gamma:G] \cdot \sharp (\mathrm{Pic}(B)) < \infty$, this implies that $\mathrm{Tor}\left( G^{\mathrm{ab}}\right)$ is finite, which contradicts statement~(c).
\end{proof}

\begin{remark}
 Note that it follows from the proof of statement~(c) that each abelian group $\tilde{G}_{u,v}$ is isomorphic to a $\mathbb{F}_p$-vector space of denumerable dimension. 
 Thus, statements~(a) implies that $\mathrm{Tor}\left(G^{\mathrm{ab}}\right)$ is isomorphic to a $\mathbb{F}_p$-vector space of denumerable dimension.
However, statement~(c) is stronger as it gives an explicit isomorphism.
\end{remark}

\begin{example}\label{ex 1}
Let $J$ be an ideal of $B$ and let $\Gamma_J= \Gamma \cap \mathrm{ker}\left(\mathrm{SL}_3(B) \to \mathrm{SL}_3(B/J)\right)$. This is a principal congruence subgroup of $\Gamma$. It follows from \cite[Lemma 3.3]{MasonSury} that $\Gamma_J$ is without $p'$-torsion.
Then $\mathrm{Tor}\left( \Gamma_J^{\mathrm{ab}} \right)\cong \coprod_{(u,v) \in \Sigma_0} (\tilde{\Gamma}_J)_{u,v}$. Note that, for any $b \in \lbrace a, 2a \rbrace$, $g$ belongs to $g_{u,v}^{-1} \mathcal{U}_b(k) g_{u,v} \cap \Gamma_J$ if and only if $g=g_{u,v}^{-1} \mathrm{u}_a(x,y) g_{u,v}$, for some $x,\overline{x},y \in J$, and $x=0$ if $b=2a$. This implies that $(\tilde{\Gamma}_J)_{u,v}$ is isomorphic to $J$.
Thus $\mathrm{Tor}\left(\Gamma_J^{\mathrm{ab}}\right)$ is isomorphic to the product of $\sharp(\Sigma_0)$ copies of $J$.
\end{example}

\begin{remark}
Since $H_1(G, \mathbb{Q})=H_1(G, \mathbb{Z})\otimes_{\mathbb{Z}} \mathbb{Q}$, the free rank of $G^{\mathrm{ab}}$ equals $d=\mathrm{dim}_{\mathbb{Q}}(H_1(G, \mathbb{Q}))$. 
Then, it follows the same argument as \cite[Lemma 1.1]{MasonSchweizer} that $d$ equals the free rank of $\pi_1(G \backslash X)$.
In particular, $d=0$ exactly when $G \backslash X$ is a tree.

For instance, let $A=\mathbb{F}[t]$, $B=\mathbb{F}[\sqrt{t}]$ and $\Gamma=\mathcal{G}(A)$.
Let $\Gamma_J$ be the principal congruence subgroup defined by $J=\sqrt{t} B$.
Adapting \cite[Ch.I, \S 1.6, Exc. 5]{S} to the context of $\mathcal{G}$ and using \cite[Theorem 2.4]{Arenas}, we can prove that $\Gamma_J \backslash X$ is a tree.
Thus $d=0$, even when $\chi(\Gamma_J)<0$.
This shows that the upper bound in Prop.~\ref{prop abelianization of G} (b) is not optimal in general.
\end{remark}

In topology, the Euler-Poincar\'e characteristic of a variety equals the alternated sum of the rank of their homology groups. Thus, the next result, which directly follows from Theorem~\ref{teo relative homology}, can be considered as an analogous of the topological result in the context of arithmetical subgroups.

\begin{corollary}
Let $G$ be a subgroup of $\Gamma$ without $p'$-torsion.
Then, the Euler-Poincar\'e characteristic $\chi(G)$ of $G$ equals each one of the following expressions:
\begin{enumerate}
\item[(1)] $\sum_{i=1}^{\infty} \mathrm{rank} \,  H_i(G \text{ mod } G_{\sigma}, \mathbb{Z})$,
\item[(2)] $-\mathrm{rank} \, H_1( G \text{ mod }G_{\sigma}, \mathbb{Z})$,
\item[(3)] $\sum_{i=1}^{\infty} \mathrm{dim} \,  H_i(G \text{ mod } G_{\sigma}, \mathbb{Q})$,
\item[(4)] $-\mathrm{dim} \, H_1( G \text{ mod }G_{\sigma}, \mathbb{Z})$.
\end{enumerate}
\end{corollary}

\section*{Acknowledgements} I express my gratitude with Anid-Conicyt by the Postdoctoral fellowship $[74220027]$.

\end{document}